\documentclass[11pt,english]{amsart}
\usepackage{amsfonts, amssymb, amsmath, amsthm, latexsym, enumerate, epsfig, color, mathrsfs, fancyhdr, pdfsync, eurosym, endnotes, stmaryrd, bm}
\usepackage[all]{xy}

\usepackage[english]{babel}

\usepackage[mac]{inputenc}
\usepackage[T1]{fontenc}
\usepackage{bbm}

\usepackage{hyperref}
\allowdisplaybreaks

\theoremstyle{plain}
\newtheorem{theorem}{Theorem}[section]
\newtheorem{lemma}[theorem]{Lemma}
\newtheorem{proposition}[theorem]{Proposition}
\newtheorem{corollary}[theorem]{Corollary}
\newtheorem{definition}[theorem]{Definition}
\newtheorem{notation}[theorem]{Notation}
\newtheorem{defn}[theorem]{Definition}
\newtheorem{thm}[theorem]{Theorem}

\theoremstyle{remark}
\newtheorem{remark}[theorem]{Remark}
\newtheorem{rmk}[theorem]{Remark}

\newtheorem{example}[theorem]{Example}

\numberwithin{equation}{section}

\newcommand\Rc{\mathcal{R}}

\newcommand\Os{\mathscr{O}}
\newcommand\Hs{\mathscr{H}}

\renewcommand\le{\leqslant}
\renewcommand\ge{\geqslant}

\DeclareMathOperator*{\bigast}{\raisebox{-0.6ex}{\scalebox{2.5}{$\ast$}}}

\def\A{{\mathbb{A}}}

\def\R{{\mathbb{R}}}
\def\N{{\mathbb{N}}}
\def\Z{{\mathbb{Z}}}
\def\sH{{\mathscr{H}}}
\def\sC{{\mathscr{C}}}
\def\sE{{\mathscr{E}}}
\def\sF{{\mathscr{F}}}
\def\sO{{\mathscr{O}}}

\def\ks{\mathfrak{s}}
\def\ki{\mathfrak{i}}

\def\fb{{\mathbbm{f}}}

\def\cB{{\mathcal B}}\def\cD{{\mathcal D}}\def\cE{{\mathcal E}}\def\cM{{\mathcal M}}\def\cN{{\mathcal N}}\def\cO{{\mathcal O}}\def\cM{{\mathcal{M}}}\def\cR{{\mathcal R}}

\def\cf{\emph{cf.}\;}\def\ie{\emph{i.e.}\;}

\def\iso{\xrightarrow{\ \sim\ }}

\def\wt{\what{\otimes}}

\def\vphi{\varphi}

\def\vt{{\vec{t}}}
\def\vv{\vec{v}}

\def\ena{{(\sE,\nabla)}}
\def\fna{{(\sF,\nabla)}}

\def\Irr{{\rm Irr}}

\def\wtilde{\widetilde}
\def\what{\widehat}

\def\Spec{{\rm Spec\,}}

\def\ker{{\rm Ker\,}}
\def\sp{\mathrm{sp}}
\def\hyp{\mathrm{hyp}}
\def\Frac{\mathrm{Frac}}
\def\ho{\hat{\otimes}}

\begin{document}
\setlength{\baselineskip}{0.55cm}	

\title{Pushforwards of $p$-adic differential equations}

\author{Velibor Bojkovi\'c }
\email{bojkovic@math.unipd.it}
\address{Dipartimento di matematica Tullio Levi-Civita, Universit\`a di Padova, Via Trieste 63, 35121 Padova}

\author{J\'er\^ome Poineau}
\email{jerome.poineau@unicaen.fr}
\address{Laboratoire de math\'ematiques Nicolas Oresme, Universit\'e de Caen, BP 5186, F-14032 Caen Cedex}

\thanks{The second author was partially supported by the ANR project ``GLOBES'': ANR-12-JS01-0007-01 and ERC Starting Grant ``TOSSIBERG'': 637027.}
\date{\today}


\begin{abstract}
Given a differential equation on a smooth $p$-adic analytic curve, one may construct a new one by pushing forward by an \'etale morphism. The main result of the paper provides an explicit formula that relates the radii of convergence of the solutions of the two differential equations using invariants coming from the topological behavior of the morphism. We recover as particular cases the known formulas for Frobenius morphisms and tame morphisms.

As an application, we show that the radii of convergence of the pushforward of the trivial differential equation at a point coincide with the upper ramification jumps of the extension of the residue field of the point given by the morphism. We also derive a general formula computing the Laplacian of the height of the Newton polygon of a $p$-adic differential equation.

\end{abstract}

\maketitle
\tableofcontents

\section*{Introduction}

 The theory of $p$-adic differential equations has seen quite a progress in the last years, not only in better understanding of the classical situation --- equations with coefficients analytic functions on the affine line --- but also in extending many classical results to a more general setting, where coefficients are now analytic functions on curves of higher genus (see \cite{Bal10,Pul,Poi-Pul2,Kedlayalocalglobal}). Berkovich theory of analytic spaces provided a fertile ground to do so, as, in this setting, the generic point, which is one of the most important notions in the classical theory,  becomes a true point of the underlying space.

 In order to extend classical results to $p$-adic differential equations on higher genus curves, one of the dominating ideas is to pushforward the situation into the affine line by a suitable map so that the change of invariants (such as the radii of convergence of solutions) can be tracked explicitly. To be more precise, let us fix an algebraically closed complete nonarchimedean nontrivially valued field $k$ of characteristic 0. If we are given a quasi-smooth Berkovich curve $Y$, a differential equation $\ena$ on $Y$ and a point $y\in Y$, in order to understand the behavior of $\ena$ around $y$, we look for a finite \'etale map $\vphi:Y\to X$, where now $X$ is a domain in $\A^1_k$ and study the equation $\vphi_*\ena$ around the point $x=\vphi(y)$. If we can find a map $\vphi$ such that we can precisely say how the invariants of $\ena$ change when passing to $\vphi_*\ena$, we may substitute our original problem with, presumably, an easier one.
 
 A celebrated example of such a map $\vphi$ is the Frobenius endomorphism of the affine line $\A^1_k$, when~$k$ has positive residue characteristic~$p$, that acts on the elements of $k$ by simply raising to the power $p$. In this case, one can describe precisely how the radii of convergence get modified (\cf \cite[Theorem 10.5.1]{pde}), and this has been vastly exploited in bridging two relatively well understood classes of $p$-adic differential equations: the one where the radius of convergence at a given point is equal to 1, and the one where all the local solutions have a small radius of convergence (smaller than $p^{1/(p-1)}$). Unfortunately, Frobenius pushforward is only applicable for differential equations on domains of $\A^1_k$ 
(although an extension to higher genus curves has been announced in \cite{Bal-Ked}).

Another example, quite orthogonal to the former, where we can explicitly say how the radii of convergence of solutions change when passing from~$y$ to $x=\varphi(y)$ is the one where the degree of the extension $[\sH(y):\sH(x)]$ is prime to the residue characteristic of~$k$. It has been extensively used in~\cite{Poi-Pul2} to establish continuity properties of radii of convergence of solutions of differential equations on curves of higher genus by pushing forward to the line, where such properties had previously been established in \cite{Pul}. A similar example has been used in~\cite{Poi-Pul3} to prove harmonicity properties of the radii of convergence, subject to the quite unnatural assumption that the residue characteristic is different from~2 (in order to be able to impose the condition on the degree of the extensions in all directions out of a given point).
 
 Of course, there is a separate intrinsic interest in describing the invariants of the pushforward of a differential equation by a finite \'etale map $\vphi$ as above (for $X$ not necessarily in $\A^1_k$) in terms of the original equation, especially in describing the change of irregularity. To the best of our knowledge, very little is known so far.
 
 \bigbreak
 
  Another aspect is the relation between $p$-adic differential equations and ramification of morphisms. It is long believed that there is a quite close relation between the pushforward of the \textit{constant} differential equation by a morphism~$\varphi$ (that is $\vphi_*(\sO_Y,d_Y)$) and ramification properties of~$\vphi$. Perhaps a first paper exploring this relation is \cite{BaldaMilano}, where Francesco Baldassarri studied the bound on the radius of convergence in terms of the \textit{different} of the morphism at a point, consequently deducing a variant of the $p$-adic Rolle theorem. The recent paper~\cite{TemkinHerbrand} by Michael Temkin introduced a refined invariant of a morphism at a non-rational point $y\in Y$, called a profile function (roughly speaking, it is a map of the real unit interval to itself that describes the morphism locally around $y$). Furthermore, he developed a ramification theory for Galois extensions of one-dimensional analytic fields, introducing Herbrand function for such extensions. One of the main results in \textit{ibid.} is that the Herbrand function of the extension $\sH(y)/\sH(\varphi(y))$ coincides with the profile function of $\vphi$ at $y$.
  
\bigbreak
 
 This is where the goal of the present article fits in: given a finite \'etale morphism \mbox{$\vphi:Y\to X$} of quasi-smooth $k$-analytic curves, points $y\in Y$ and $x=\vphi(y)\in X$ and a differential equation $\ena$ on $Y$, describe the radii of convergence of solutions of $\vphi_*\ena$ at $x$ in terms of the radii of convergence of solutions of $\ena$ at $y$ and the profile function of the morphism~$\vphi$ at~$y$. The goal is achieved in Corollary \ref{cor:radial over discs}, where $Y$ and $X$ are open unit discs and $y$ is taken to be a rational point, and in Theorem \ref{push general} for general $Y$, $X$ and $y$. Along the way, we provide simplified formulas for some special cases of $\ena$ such as the constant connection. All the known instances of our results, such as the ones described above, are summarized in Corollary \ref{cor:specialcases} and we hope to convince the reader that they can be easily recovered using our arguments.

 After seeing the first version of this article, Michael Temkin suggested that our pushforward formula should have a form in which the transitivity property with respect to composition of morphisms looks more visible, similarly to the case of a profile function of a morphism. This was the reason to introduce the profile of a differential equation (at a point) at the end of the third section. Then the desired transitivity property becomes transparent  in Theorem \ref{thm:profile version}. We would like to thank Michael for his remarks concerning this.
 
  We restrict ourselves to applications in two (apparently) different directions, namely ramification of one-dimensional analytic fields on the one hand and irregularities and Laplacians in the area of $p$-adic differential equations on the other hand. In the former case, we establish a close (expected) relation between the Herbrand function of a Galois extension of one-dimensional analytic fields $\sH(y)/\sH(x)$ (for non-rational points~$x$ and~$y$) and the radii of convergence at $x$ of local sections of the pushforward of the constant connnection $\vphi_*(\sO,d)$ (Corollary \ref{cor:Herbrand}). In the latter case, we provide a formula for the change of irregularity and Laplacians when passing from $\ena$ to $\vphi_*\ena$, and use it to prove harmonicity properties of convergence polygons on arbitrary curves (Corollary~\ref{cor:partialheight}).  
  
\bigbreak

The paper is structured as follows. In the first section we recall the properties of the profile of a morphism $\vphi$ at a (non-rational) point $y\in Y$. Along the way, we provide an alternative definition of the profile function based on the structure of the fibers of projection maps coming from extensions of scalars.

 The second section is reserved for recalling the main definitions and results concerning $p$-adic differential equations on Berkovich curves. 
 
The third and fourth sections are the core of the paper. In the third section, we provide the main results about the radii of pushforward of differential equations, while the fourth contains applications of our main formulas, as described above.

%The paper is structured as follows. In the second section we recall the properties of the profile of a morphism $\vphi$ at a (non-rational) point $y\in Y$. Along the way, we provide an alternative definition of the profile function based on the structure of the fibers of projection maps coming from extensions of scalars. 
%
%
% The third section is reserved for recalling the main definitions and results concerning $p$-adic differential equations on Berkovich curves. 
% 
%The fourth and fifth sections are the core of the paper. In the fourth section, we provide the main results about the radii of pushforward of differential equations, while the fifth contains applications of our main formulas, as described above. 

\section{Finite morphisms between quasi-smooth $k$-analytic curves}
Let $k$ be an algebraically closed field, complete with respect to a non-archimedean and non-trivial valuation and of characteristic $0$. The corresponding norm on the field is denoted by $|\cdot|$ and the residue field is denoted by $\tilde{k}$. 

All the curves throughout the paper will be $k$-analytic in the sense of Berkovich. In addition, we will assume that they are quasi-smooth (rig-smooth) and unless otherwise stated we will assume that they are connected too.

\subsection{Quasi-smooth Berkovich curves}

We will use for granted the construction of the $k$-analytic affine line $\A^1_k$ from \cite[Section 4.2.]{Ber90}.

For $a\in k$ and $r\in [0,\infty)$, we denote by $D(a,r)$ (resp. $D(a,r^-)$ if $r>0$) the Berkovich closed (resp. open) disc centered at the point $a$ and of radius $r$. Similarly, for $a\in k$ and $r_1\ge r_2\in (0,\infty)$ (resp. $r_1> r_2\in(0,\infty)$) we denote by $A[a;r_2,r_1]$ (resp. $A(a;r_2,r_1)$) the Berkovich closed (resp. open) annulus centered at the point $a$ and of inner radius $r_2$ and outer radius $r_1$. The Shilov point of the disc $D(a,r)$ will be denoted by $\eta_{a,r}$ or $\eta_r$ if $a=0$.

Note that the disc $D(a,r)$ is strictly $k$-affinoid (we will simply say strict) if and only if $r\in |k^*|$. By abuse of notation we will say that an open disc $D(a,r^-)$ is strict if $r\in |k^*|$. In this case it is isomorphic to the open unit disc $D(0,1^-)$. Similarly, we say that the open annulus $A(a;r_2,r_1)$ is strict if $r_1,r_2\in |k^*|$; in this case, it is isomorphic to the normalized annulus $A(0;r,1)$, where $r = r_{2}/r_{1}\in |k^*|$. 

In general, given a $k$-analytic curve $X$, an analytic domain of $X$ isomorphic to an open disc (resp. strict open disc, open annulus, strict open annulus) is called an open disc (resp. strict open disc, open annulus, strict open annulus). If $D$ (resp. $A$) is a strict open disc (resp. strict open annulus), then any isomorphism $T:D\iso D(0,1^-)$ (resp. $T:A\iso A(0;r,1)$) will be called a normalizing coordinate on $D$ (resp. a normalizing coordinate on $A$).

%In general, given a $k$-analytic curve $X$, an analytic domain of $X$ isomorphic to an (strict) open disc (resp. (strict) open annulus) is called an (strict) open disc (resp. (strict) open annulus). If $D$ (resp. $A$) is a strict open disc (resp. strict open annulus), then any isomorphism $T:D\iso D(0,1)$ (resp. $T:A\iso A(0;r,1)$) will be called a normalizing coordinate on $D$ (resp. a normalizing coordinate on $A$).

\medbreak

The skeleton $S(A)$ of an annulus $A$ is the complement of all the open discs in $A$. For example, if $A=A(a;r_2,r_1)$, then $S(A)=\{\eta_{a,r}\mid r\in (r_2,r_1)\}$ and it can be naturally identified with the interval $(r,1)$, where $r=r_{2}/r_{1}$. This interval only depends on~$A$ and not on the choices of~$r_{2}$ and~$r_{1}$. In particular, $S(A)$ is naturally endowed with a metric.

It follows from their explicit descriptions that discs and annuli are endowed with (real) tree structures and that each interval in them containing no type~1 point naturally carries a metric modeled on that on the skeleton of an annulus above.

More generally, it follows from the semistable reduction theorem that, up to removing a locally finite set, any quasi-smooth $k$-analytic curve is a disjoint union of discs and annuli. As a consequence, quasi-smooth $k$-analytic curves are endowed with (real) graph structures. In particular, it makes sense to speak of intervals in them and, as before, any interval containing no type~1 point naturally carries a metric.

\medbreak

Now, let $x$ be a point in~$X$. The set of germs of intervals in $X$ with an endpoint in $x$ is denoted by $T_xX$ and is called the set of tangent directions (or tangent space) at $x$. An element of $T_xX$ is also called a branch. 

\medbreak

Let $K/k$ be a complete valued extension of $k$ and let us put $X_K:=X\wt_k K$ (see \cite[2.1]{Ber90} for the construction of $X\wt_k K$). Then, we have a natural projection map $\pi_{K/k}:X_K\to X$ (sometimes denoted by $\pi_K$ if $k$ is understood from the context). If $x$ is a point in~$X$, then each $K$-rational point in the set $\pi_{K/k}^{-1}(x)$ will be called a $K$-rational point above $x$ and we will use the notation~$x_K$ for it. For more details one may refer to \cite[Section~2.2]{Poi-Pul2}.

The previous construction and notion are meaningful even if the base field $k$ is not algebraically closed but rather just a complete valued field. However, by \cite[Corollaire~3.14]{Angie}, when $k$ is algebraically closed, each point $x\in X$ is {\em peaked} (see~\cite[Section~5.2]{Ber90}) or {\em universal} (\cite[D\'efinition 3.2]{Angie}): for each complete valued extension~$K/\sH(x)$, where $\sH(x)$ is the complete residue field of $x$, the tensor norm on $\sH(x)\wt K$ is multiplicative. In particular, $\cM(\sH(x)\wt_k K)$ contains a distinguished point, denoted by $\sigma_{K/k}(x)$ (or simply $\sigma_{K}(x)$ if~$k$ is clear from the context), which corresponds to the (multiplicative) norm on $\sH(x)\wt_k K$. By \cite[Corollaire3.7]{Angie}, the map of topological spaces $\sigma_{K/k}:X\to X_K$ is a continuous section of~$\pi_{K/k}$.

\begin{defn}
 We call the point $\sigma_{K/k}(x)$ the {\em $K$-generic point above $x$ in $X_K$}. 
\end{defn}

\subsection{Radial morphisms and profile functions}\label{radial morphisms}
We recall the notions of radial morphism and profile function introduced in \cite{TemkinHerbrand} and to which we refer for more details.

\subsubsection{Radial morphisms of open discs}\label{radial discs}

First we set up some notation that will be used later on. 

Consider the Berkovich affine line~$\A^1_{k}$ and choose a coordinate~$T$ on it. Denote by~$r$ the radius function: for each $x\in \A^1_{k}$, we have
\[r(x) := \inf_{a\in k} |(T-a)(x)|.\]

Let~$D$ be a strict open disc and let $z\in D(k)$. Choose an isomorphism $T_z:D\iso D(0,1^-)$ between~$D$ and the open unit disc in~$\A^1$ that sends~$z$ to~$0$. 

For each $\rho \in (0,1]$, set $D(z,\rho^-) = T_{z}^{-1}(D(0,\rho^-))$. Similarly, for each $\rho \in (0,1)$, set $D(z,\rho) = T_{z}^{-1}(D(0,\rho))$ and denote by~$z_{\rho}$ the unique point of the Shilov boundary of~$D(z,\rho)$. We also set $D(z,0)= \{z\}$ and $z_{0}=z$. We denote by~$l_z$ the canonical path from $z$ to the ``boundary'' of~$D$, that is 
\[l_{z} := \{z_{\rho}, \rho\in [0,1)\}.\] 
All these definitions are independent of the choice of~$T_{z}$.

We define a radius function~$r$ on~$D$ by pulling back the restriction of the radius function to~$D(0,1^-)$. It is independent of the choice of~$z$ and~$T_{z}$. 

If $\vphi:D_2\to D_1$ is a finite morphism of strict open discs, then, for each point $z\in D_2(k)$, $\vphi$ induces a continuous increasing map $f_{\vphi,z}: [0,1)\to [0,1)$ given by $\rho\mapsto r(\vphi(z_\rho))$. We sometimes extend this map to a continuous map $[0,1] \to [0,1]$ by setting $f_{\vphi,z}(1)=1$. We denote it the same way.

A closely related function is given in the next definition.

\begin{definition}\label{def:multiplicity}
Let $\vphi : Y \to X$ be a finite morphism of $k$-analytic curves. Let $y\in Y$. We define the multiplicity of~$\varphi$ at~$y$ as
\begin{equation}
n_{y} = e \cdot [\Hs(y) \colon \Hs(\vphi(y))],
\end{equation}
where $e=1$ if~$y$ has type~2, 3 or~4 and~$e$ is the unique integer such that $m_{y} = m_{\varphi(y)}^e \Os_{y}$ if~$y$ has type~1. We define the multiplicity function of~$\varphi$ by
\[n_{\varphi} \colon y\in Y \mapsto n_{y} \in \N.\]
\end{definition}

Let~$I$ be an interval in~$\R_{+}$ whose interior~$\mathring I$ is not empty. A map $f \colon I \to \R_{+}$ is called $|k^*|$-monomial if it is of the form $t \mapsto \alpha t^n$ with $\alpha\in |k^*|$ and $n\in\Z$. The integer~$n$ is then called the degree (or logarithmic slope) of~$f$. 

A map $f \colon I \to \R_{+}$ is called piecewise $|k^*|$-monomial if~$I$ may be written as a finite union of intervals $I_{i}$ with non-empty interiors such that the restriction of~$f$ to each~$I_{i}$ is $|k^*|$-monomial. For each $x \in \mathring I$, we may then define the left and right degrees of~$f$ at~$x$. We denote them respectively by $\deg^-_{f}(x)$ and $\deg^+_{f}(x)$. We say that a point $x\in \mathring I$ is a break-point of~$f$ when $\deg^-_{f}(x) \ne \deg^+_{f}(x)$. It will also be convenient to declare the points at the boundary of~$I$ to be break-points. Note that, if~$f$ is continuous, then all break-points in~$\mathring I$ belong to~$|k^*|$. 

\begin{lemma}[\protect{\cite[Lemma 2.2.5]{TemkinHerbrand}}]\label{lem:behaviourfz}
Let $\vphi:D_2\to D_1$ be a finite morphism of strict open discs. Let $z\in D_{2}(k)$. The function~$f_{\varphi,z}$ is continuous and piecewise $|k^*|$-monomial and its right degree coincides with the restriction of~$n_{\varphi}$ to~$l_{z}$ (canonically identified to $[0,1)$). 
\qed
\end{lemma}

\begin{remark}
As explained above, since~$f_{\varphi,z}$ is continuous and piecewise $|k^*|$-monomial, the break-points of~$f_{\varphi,z}$ belong to $|k| \cap [0,1)$. In other words, the points where the restriction of~$n_{\varphi}$ to~$l_{z}$ is not locally constant are of type~2.
\end{remark}

Another related function is the following.

\begin{definition}\label{N function} 
Let $\vphi:D_2\to D_1$ be a finite morphism of strict open discs. For $z\in D_1(k)$ and $s \in (0,1]$, we denote by~$N_{\varphi,z}(s)$ the number of connected components of the set $\vphi^{-1}(D(z,s^-))$.
\end{definition}

\begin{defn}[\protect{\cite[2.3.2]{TemkinHerbrand}}]
Let $\vphi:D_2\to D_1$ be a finite morphism of strict open discs. We say that $\vphi$ is a radial morphism if the induced functions $f_{\vphi,z}$ are equal for all $z \in D_{2}(k)$. In this case, the function $f_\vphi:=f_{\vphi,z}$ is called the profile function of~$\vphi$.
\end{defn}

The following result is well-known and can easily be proved by using a coordinate representation of the morphism and studying the associated valuation polygon. 
%studying valuation polygon and coordinate representation of the morphism. 

\begin{lemma}\label{lem:imagepreimage}
Let $\vphi$ be a finite morphism between open discs. Then, the image of a subdisc is a subdisc and the preimage of a subdisc is a finite union of subdiscs. Furthermore, if $\vphi$ is radial, all the preimages of a given subdisc have the same radius.
\qed
\end{lemma}

Let us now explain the relationship with~$N_{\varphi,z}$.

\begin{lemma}\label{lem:relation}
 Let $\vphi:D_2\to D_1$ be a finite radial morphism of strict open discs of degree~$d$. The following properties hold.
 \begin{enumerate}[(i)]
  \item For each $x\in D_{1}$ and each $y\in \varphi^{-1}(x)$, we have $\# \varphi^{-1}(x) = d/n_{\varphi}(y)$.
  \item For each $z\in D_{1}(k)$ and each $r\in (0,1)$, we have 
  \begin{equation}
  N_{\varphi,z}(r) = \min (\{\# \varphi^{-1}(z_{\rho}) \mid \rho \in [0,r)\}) = \frac{d}{\deg^-_{f_{\varphi}}(f_{\varphi}^{-1}(r))}.
  \end{equation}
  \item The function~$N_{\varphi,z}$ is independent of $z\in D_{1}(k)$ and  its points of discontinuity are exactly the break-points of~$f^{-1}_{\varphi}$. 
  \end{enumerate}
\end{lemma}
\begin{proof}
Let~$x \in D_{1}$. Since the morphism~$\varphi$ is radial, all points over~$x$ have the same radius. By Lemma~\ref{lem:behaviourfz}, $n_{\varphi}$ then takes the same value on all of them. Point (i) follows.

Let us now prove point~(ii). Let $z\in D_{1}(k)$ and let $r\in (0,1)$. By Lemma~\ref{lem:imagepreimage}, the preimage of the open disc $D(z,r)^-$ is a disjoint union of open discs. Such an open disc contains only one point over~$z_{\rho}$ for $\rho \in (0,r)$ close enough to~$r$. The first equality follows. The second one follows from point~(i) and Lemma~\ref{lem:behaviourfz}.

Point~(iii) follows from point~(ii).
\end{proof}

As a consequence, when the morphism~$\varphi$ is radial, we often write~$N_{\varphi}$ instead of~$N_{\varphi,z}$.

\begin{lemma}\label{lem:image of profile}
 Let $\vphi:D_2\to D_1$ be a finite radial morphism of open unit discs with profile function $f$. Let $0=t_0<t_1<\dots<t_n=1$ be the break-points of~$f$ and, for $i=1,\dots,n$, let~$d_i$ denote the degree of~$f$ over the interval $(t_{i-1},t_i)$. Let $z\in D_2(k)$ and let $T_{\vphi(z)}=\vphi(T_z)=\sum_{i\geq 0}a_iT^i_z$ be the coordinate representation of $\vphi$ with respect to $T_z$ and $T_{\vphi(z)}$. Then, for $i=1,\dots,n$, we have
 \begin{equation}\label{coefficient relation}
  |a_{d_{i}}|=|a_{d_{1}}| \prod_{j=1}^{i-1} t_j^{d_j-d_{j+1}}.
  \end{equation}
  For each $m=0,\dots,n-1$ and each $r \in (t_{m},t_{m+1}]$, we have
  \begin{equation}\label{profileimage}
  f(r)=|a_{d_{1}}|r^{d_{m+1}}\prod_{j=1}^mt_j^{d_{j}-d_{j+1}}.
\end{equation}
In particular, for $i=1,\dots,n$, we have
\begin{equation}\label{profileimage2}
f(t_i)=|a_{d_{1}}|t_i^{d_i}\prod_{j=1}^{i-1}t_j^{d_j-d_{j+1}}.
\end{equation}
\end{lemma}
\begin{proof}
 For $i=0,\dots,n-1$ and $\rho\in (t_{i},t_{i+1})$, we have $f(\rho)=|\vphi(T_z)|_{\rho}=\max|a_i|\rho^i=|a_{d_{i+1}}|\rho^{d_{i+1}}$, because $f$ is $|k^*|$-monomial over $(t_i,t_{i+1})$ of degree $d_{i+1}$. 
 
 Because of the continuity of $|\vphi(T_z)|_{(\cdot)}$ at~$t_{i}$, for $i=1,\dots,n-1$, we deduce that $|a_{d_i}|t_i^{d_i}=|a_{d_{i+1}}|t_i^{d_{i+1}}$, which implies $|a_{d_{i+1}}|=|a_{d_i}|t_i^{d_i-d_{i+1}}$. Equation \eqref{coefficient relation} then follows by induction.
 
 Equation \eqref{profileimage} follows by noting that $f(r)=|\vphi(T_z)|_r=|a_{d_{m+1}}|r^{d_{m+1}}$ and using~\eqref{coefficient relation}. 
\end{proof}

\subsubsection{Profile functions in the general case}\label{sec:profile}

Let $\varphi \colon Y \to X$ be a finite morphism between quasi-smooth $k$-analytic curves. We denote by~$Y^{(2)}$ (resp.~$Y^\hyp$) the set of type~2 (resp. non-rigid) points of~$Y$. 

For every point~$y$ in~$Y$, we denote by~$\cD_{y}$ the set of connected components of~$Y \setminus\{y\}$ whose closure contains~$y$.

\begin{theorem}\label{thm:radializingskeleton}
Let~$y \in Y^{(2)}$. There exists a finite subset~$F_{y}$ of~$\cD_{y}$ and a map $^T\!f_{\varphi}^y: [0,1] \to [0,1]$ such that, for each~$D$ in $\cD_{y} \setminus F_{y}$, $D$ and~$\varphi(D)$ are strict open discs and the induced morphism $D \to \varphi(D)$ is radial with profile function~$^T\!f_{\varphi}^y$.
\end{theorem}
\begin{proof}
It is a consequence of the existence of a so-called ``radializing skeleton'', see \cite[Theorem 3.3.11]{TemkinHerbrand}. Note that this reference only handles compact strictly $k$-analytic curves, but this easily implies the general case.
\end{proof}

In \cite[Section~3.4.1]{TemkinHerbrand}, for each $y \in Y^{(2)}$, Michael Temkin defines the \textit{profile function} of~$\varphi$ at~$y$ to be~$^T\!f_{\varphi}^y$.

Putting together all those function gives a map $^T\!f_{\varphi}^{(2)} : Y^{(2)} \times [0,1] \to [0,1]$. By \cite[Theorem 3.4.8]{TemkinHerbrand}, it extends uniquely to a map $^T\!f_{\varphi} : Y^\hyp \times [0,1] \to [0,1]$ that is piecewise $|k^\times|$-monomial (in an appropriate sense) and continuous on each interval inside~$Y^\hyp$.

\begin{lemma}\label{lem:fiber}
Let $y\in Y^\hyp$ and set $x:=\vphi(y)$. Let~$K$ be an algebraically closed complete valued extension of~$\sH(y)$. Then, the connected components of $\pi^{-1}_{K}(y) \setminus\{\sigma_{K}(y)\}$ and $\pi^{-1}_{K}(x) \setminus\{\sigma_{K}(x)\}$ are strict open discs.

Let~$D$ be a connected component of $\pi^{-1}_{K}(x) \setminus\{\sigma_{K}(x)\}$. The set $\varphi_{K}^{-1}(D) \cap \pi^{-1}_{K}(y)$ is a disjoint union of connected components $D'_{1},\dotsc,D'_{\ks}$ of $\pi^{-1}_{K}(y) \setminus\{\sigma_{K}(y)\}$ and, for each $j=1,\dotsc,\ks$, the induced morphism $\psi_{j} \colon D'_{j} \to D$ is radial. Moreover, the integer~$\ks$, the functions~$f_{\psi_{j}}$ and~$N_{\psi_{j}}$ and the degree of~$\psi_{j}$ only depend on~$y$ (\textit{i.e.} not on~$K$, $D$ or~$j$).
\end{lemma}
\begin{proof}
Up to shrinking~$Y$ and~$X$ around~$y$ and~$x$, we may assume that $\vphi^{-1}(x)=\{y\}$. 

By \cite[Theorem~2.15]{Poi-Pul2}, the connected components of $\pi^{-1}_{K}(y) \setminus\{\sigma_{K}(y)\}$ are isomorphic to open discs with boundary~$\{y_{K}\}$. Since the point~$\sigma_{K}(y)$ has type~2, those discs are strict. The same result holds for $\pi^{-1}_{K}(x) \setminus\{\sigma_{K}(x)\}$. 

Let~$D$ be a connected component of $\pi^{-1}_{K}(x) \setminus\{\sigma_{K}(x)\}$. It follows from the continuity of~$\vphi_{K}$ that~$\varphi_{K}^{-1}(D)$ is a disjoint union of connected components $D'_{1},\dotsc,D'_{\ks}$ of $\pi^{-1}_{K}(y) \setminus\{\sigma_{K}(y)\}$. 

To prove that~$\psi_{j}$ is radial, we may extend the scalars and assume that~$K$ is maximally complete. By~\cite[Corollary~2.20]{Poi-Pul2}, the group $\textrm{Gal}^c(K/k)$ of continuous $k$-linear automorphisms of~$K$ acts transitively on the set of $K$-rational points of~$D'_{j}$ and the result follows. 

To prove that~$\ks$, $f_{\psi_{j}}$, $N_{\psi_{j}}$ and $\deg(\psi_{j})$ do not depend on~$D$ and~$j$, we may also assume that~$K$ is maximally complete. By \textit{ibid.}, $\textrm{Gal}^c(K/k)$ acts transitively on the set of connected components of $\pi^{-1}_{K}(y) \setminus\{\sigma_{K}(y)\}$ and $\pi^{-1}_{K}(x) \setminus\{\sigma_{K}(x)\}$ and the result follows.

It is clear that $\ks$, $f_{\psi_{j}}$, $N_{\psi_{j}}$ and $\deg(\psi_{j})$ are independent of~$K$ since any two complete valued extensions of~$\Hs(y)$ may be embedded into a common one.
\end{proof}

\begin{remark}
The statement still holds, with the same proof, if~$k$ has positive characteristic.
\end{remark}

\begin{notation}\label{not:fyNy}
In the setting of Lemma~\ref{lem:fiber}, we set $\ks_{\vphi}^y := \ks$, $f_{\varphi}^y := f_{\psi_{1}}$, $N_{\varphi}^y := N_{\psi_{1}}$ and $\ki_{\varphi}^{y} := \deg(\psi_{1})$. The map~$f_{\varphi}^y$ will be called the profile function of~$\varphi$ at~$y$. We may remove~$\vphi$ from the notation when it is clear from the context.
\end{notation}

\begin{lemma}\label{lem:profileK}
Let $y\in Y^\hyp$. Let~$L$ be an algebraically closed complete valued extension of~$\sH(y)$. Then, we have $\ks_{\varphi_{L}}^{\sigma_{L}(y)} = \ks_{\varphi}^y$, $f_{\varphi_{L}}^{\sigma_{L}(y)} = f_{\varphi}^y$, $N_{\varphi_{L}}^{\sigma_{L}(y)} = N_{\varphi}^y$ and $\ki_{\varphi_{L}}^{\sigma_{L}(y)} = \ki_{\varphi}^y$
\end{lemma}
\begin{proof}
Note that $\sH(\sigma_{L}(y))$ contains~$\sH(y)$. Let~$K$ be an algebraically closed complete valued extension of~$\sH(\sigma_{L}(y))$. The result follows easily from the fact that $\sigma_{K/k} = \sigma_{K/L} \circ \sigma_{L/k}$.
\end{proof}

We now explain how to interpret~$\ks^y_{\varphi}$ as a separable degree and~$\ki^y_{\varphi}$ as an inseparable degree. Recall that, for each $y\in X^{(2)}$, the residue field~$\wtilde{\sH(y)}$ is of transcendence degree~1 over~$\tilde k$. In particular, it is the function field of a well-defined smooth connected projective curve over~$\tilde{k}$. We denote it by~$\sC_{y}$ and call it the \textit{residue curve} at~$y$.

\begin{lemma}\label{lem:CxK}
Let $y\in Y^{(2)}$. For every complete valued extension~$K$ of~$k$, we have a canonical isomorphism 
\[\sC_{y} \otimes_{\tilde k} \tilde K \iso \sC_{\sigma_{K}(y)}.\] 
\end{lemma}
\begin{proof}
We may assume that~$Y$ is reduced since reducing it leaves~$\sH(y)$ unchanged. Let $W = \cM(\cB)$ be a strictly $k$-affinoid domain of~$Y$ whose Shilov boundary is exactly~$\{y\}$. By \cite[Proposition~2.4.4 (iii)]{Ber90}, $\tilde{W} = \Spec(\tilde{\cB})$ is irreducible and its generic point is the reduction~$\tilde{y}$ of~$y$. Since~$k$ is algebraically closed, by~\cite[Proposition~6.2.1/4 (ii)]{BGR}, we have $|\cB|_{\textrm{sup}} = |k|$, hence, by \cite[Proposition~2.4.4 (ii)]{Ber90}, an isomorphism
\[\Frac(\tilde{\cB}) \iso \tilde{k}(\tilde{y}) \iso \wtilde{\sH(y)}.\]

Since~$k$ is algebraically closed, it is stable (see~\cite[Section~3.6]{BGR} for definitions). Since~$W$ is reduced, by~\cite[Theorem~6.4.3/1]{BGR}, $W$ is distinguished. Since~$\tilde{k}$ is also algebraically closed, $\tilde{W}$ is geometrically integral. In particular, $\tilde{\cB} \otimes_{\tilde{k}}\tilde{K}$ is a domain. By \cite[Proposition~5.2.5 and its proof]{Ber90}, the point~$\sigma_{K}(y)$ is the unique point in the Shilov boundary of $W_{K} = \cM(\cB\ho_{k}K)$, hence, by \cite[Proposition~2.4.4]{Ber90} again, we have an isomorphism
\[\Frac(\wtilde{\cB\ho_{k}K})  \iso \wtilde{\sH(\sigma_{K}(y))}.\]
By \cite[\S 6, Satz~4]{Orthonormalbasen}, we have an isomorphism $\wtilde{\cB\ho_{k} K} \simeq \tilde{\cB}\otimes_{\tilde{k}}\tilde{K}$ and the result follows.
\end{proof}

Note that, for each $y\in Y^{(2)}$, $\varphi$ naturally induces a morphism $\sH(\varphi(y)) \to \sH(y)$, hence a morphism $\wtilde{\sH(\varphi(y))} \to \wtilde{\sH(y)}$ and finally a morphism $\tilde \varphi_{y}: \sC_{y} \to \sC_{\varphi(y)}$ between the residue curves.

Recall that, since~$k$ is algebraically closed, for each $y\in X^{(2)}$, $\sH(y)$ is stable (see~\cite[Corollary~6.3.6]{stablemodification} or \cite[Th\'eor\`eme~3.3.16]{Duc-book}), hence we have
\[[\sH(y):\sH(\varphi(y))] = [\wtilde{\sH(y)}:\wtilde{\sH(\varphi(y))}] = \deg(\tilde{\varphi}_{y}).\]

\begin{lemma}\label{lem:phiyK}
Let~$y \in Y^{(2)}$ and set $x:=\varphi(y)$. Let~$K$ be a complete valued extension of~$k$. Then, the base-change of $\tilde{\varphi}_{y} : \sC_{y} \to \sC_{x}$ to~$\tilde{K}$  coincides with $\wtilde{(\varphi_{K})}_{\sigma_{K}(y)} : \sC_{\sigma_{K}(y)} \to \sC_{\sigma_{K}(x)}$.

Let~$D$ be a connected component of~$\pi_{K}^{-1}(x)\setminus\{\sigma_{K}(x)\}$. Then, the number of connected components of~$\varphi_{K}^{-1}(D) \cap \pi_{K}^{-1}(y)$ is equal to the separable degree of~$\tilde{\varphi}_{y}$, \ie to the separable degree of~$\wtilde{\sH(y)}/\wtilde{\sH(x)}$.

For each connected component~$D'$ of ~$\varphi_{K}^{-1}(D) \cap \pi_{K}^{-1}(y)$, the degree of~$(\varphi_{K})_{|D'}$ is equal to the inseparable degree of $\tilde{\varphi}_{y}$, \ie to the inseparable degree of ~$\wtilde{\sH(y)}/\wtilde{\sH(x)}$.
\end{lemma}
\begin{proof}
We may assume that~$Y$ and~$X$ are reduced. As in the proof of Lemma~\ref{lem:CxK}, we consider a strictly $k$-affinoid domain~$W$ of~$Y$ whose Shilov boundary is exactly~$\{y\}$ and a strictly $k$-affinoid domain~$V$ of~$X$ whose Shilov boundary is exactly~$\{x\}$. We may assume that $\varphi^{-1}(x)=\{y\}$ and $\varphi^{-1}(V) \subseteq W$, so that~$\varphi$ induces a morphism of affinoid spaces $\phi: W \to V$, hence, by reduction, a morphism of affine curves $\tilde\phi:\tilde{W} \to \tilde{V}$. The later morphism extends uniquely to a morphism of projective curves $\sC_{y} \to \sC_{x}$, which is nothing but~$\tilde{\varphi}_{x}$, by the proof of Lemma~\ref{lem:CxK}.

By base-changing to~$K$, similarly, we find a morphism of affine curves $\phi_{K}: \tilde{W_{K}} \to \tilde{V_{K}}$ that extends uniquely to $\tilde\phi_{K}: \wtilde{(\varphi_{K})}_{\sigma_{K}(y)}$. By the proof of Lemma~\ref{lem:CxK} again, we have $\tilde{W_{K}} = \tilde W\otimes_{\tilde k} \tilde K$ and $\tilde{V_{K}} = \tilde V \otimes_{\tilde k} \tilde K$ and the first part of the result follows.

\medbreak

We now consider the reduction map $V \to \tilde{V}$. Recall that the reduction~$\tilde{x}$ of~$x$ is the generic point of the curve~$\tilde{V}$. Let~$v$ be a closed point of~$\tilde V_{K}$ over~$\tilde x$. Since~$\tilde k$ is algebraically closed, $v$ is smooth, hence its preimage~$T_{v}$ by the reduction map $V_{K} \to \tilde V_{K}$ is an open disc. It follows that~$T_{v}$ is a connected component of $\pi^{-1}_{K}(x) \setminus \{\sigma_{K}(x)\}$.

The preimage of the generic point of~$\tilde{V}_{K}$ by the reduction map being~$\sigma_{K}(x)$ (see the proof of Lemma~\ref{lem:CxK}), we have just described all the connected components of $\pi^{-1}_{K}(x) \setminus \{\sigma_{K}(x)\}$. In particular, there exists a closed point~$v$ of~$\tilde V_{K}$ over~$\tilde x$ such that~$T_{v}=D$.

By a similar argument, for each point~$w$ in~$\tilde W_{K}$ over~$v$, the preimage~$T_{w}$ of~$w$ by the reduction map $W_{K} \to \tilde W_{K}$ is a connected component of $\pi^{-1}_{K}(y) \setminus \{\sigma_{K}(y)\}$. It follows that the connected components of $\varphi_{K}^{-1}(T_{v}) \cap \pi^{-1}_{K}(y)$ are in bijection with the preimages of~$v$ in~$\tilde W_{K}$. Since~$v$ lies over the generic point of~$\tilde W$, the number of its preimages by~$\tilde\phi_{K}$ is equal to the separable degree of~$\tilde\phi$, which coincides with the separable degree of~$\tilde\varphi_{y}$.

\medbreak

Finally, let~$D'$ be a connected component of~$\varphi_{K}^{-1}(D) \cap \pi_{K}^{-1}(y)$. From Lemma~\ref{lem:fiber}, it follows that $\deg((\vphi_K)_{|D'})$ does not depend on $D'$. Since we have just proven that the number of connected component of~$\varphi_{K}^{-1}(D) \cap \pi_{K}^{-1}(y)$ is equal to the separable degree of~$\tilde\varphi_{y}$, it follows that the degree of $(\vphi_K)_{|D'}$ is equal to the degree of~$\varphi_{K}$ divided by the separable degree of~$\tilde\varphi_{y}$. The result now follows from the chain of equalities
\[\deg(\varphi_{K}) = \deg(\varphi) = [\sH(y):\sH(x)] = [\wtilde{\sH(y)}:\wtilde{\sH(x)}] = \deg(\tilde\varphi_{y}).\]
\end{proof}

We now define the residual degree, and variants of it, of the morphism~$\varphi$ at a point $y\in Y^\hyp$. Let us begin with type~2 points.

\begin{definition}
Let~$y \in Y^{(2)}$. We define the residual degree (resp. residual separable degree, resp. residual inseparable degree) of~$\varphi$ at~$y$ to be the degree (resp. separable degree, resp. inseparable degree) of the morphism~$\tilde\varphi_{y}$, or equivalently of the extension $\wtilde{\sH(y)} / \wtilde{\sH(x)}$. 
\end{definition}

In order to extend the definition to arbitrary points of~$Y^\hyp$, note that, for each complete valued extension~$K$ of~$\sH(y)$, the point~$\sigma_{K}(y)$ is of type~2.

\begin{definition}\label{def:resdeghyp}
Let~$y \in Y^\hyp$. Let~$K$ be an algebraically closed complete valued extension of~$\sH(y)$. We define the residual degree (resp. residual separable degree, resp. residual inseparable degree) of~$\varphi$ at~$y$ to that of~$\varphi_{K}$ at~$\sigma_{K}(y)$. 
\end{definition}

The definition does not depend on the choice of the extension~$K$. Indeed, one may always embed two extensions~$K$ and~$K'$ into a common one~$L$ and the result then follows from Lemma~\ref{lem:phiyK} (applied first with $y = \sigma_{L}(y) = \sigma_{L/K}(\sigma_{K}(y))$, $K = L$, $k=K$ and then with $y = \sigma_{L}(y) = \sigma_{L/K'}(\sigma_{K'}(y))$, $K = L$, $k=K'$.)

Moreover, for $y\in Y^{(2)}$, the definition coincides with the former by Lemma~\ref{lem:phiyK}.

\begin{remark}\label{rem:resdeghyp}
It follows from the stability of points of type~2 (see the reminder before Lemma~\ref{lem:phiyK}) that, for each $y\in Y^\hyp$, the residual degree of~$\varphi$ at~$y$ is equal to $[\sH(y) : \sH(\varphi(y))]$, \ie the degree of~$\varphi$ at~$y$.
\end{remark}

Note that, with Definition~\ref{def:resdeghyp}, the result of Lemma~\ref{lem:phiyK} extends readily to points of type~3 and~4, which shows that the quantity~$\ks^y_{\varphi}$ from Notation~\ref{not:fyNy} is nothing but the residual separable degree of~$\varphi$ at~$y$ while $\ki^y_{\varphi}$ (in the notation of Lemma \ref{lem:fiber}) is the residual inseparable degree of $\vphi$ at $y$.

\begin{lemma}\label{lem:phiyK34}
Let $y\in Y^\hyp$ and set $x:=\varphi(y)$. Let~$K$ be an algebraically closed complete valued extension of~$k$. Then, for each connected component~$D$ of~$\pi_{K}^{-1}(x)\setminus\{\sigma_{K}(x)\}$, the number of connected components of~$\varphi_{K}^{-1}(D) \cap \pi_{K}^{-1}(y)$ is equal to the residual separable degree of~$\varphi$ at~$y$ and, for each connected component~$D'$ of~$\varphi_{K}^{-1}(D) \cap \pi_{K}^{-1}(y)$, the degree of~$(\varphi_{K})_{|D'}$
is equal to the residual inseparable degree of $\vphi$ at $y$. 
\qed
\end{lemma}

\begin{proposition}\label{prop:comparison}
For each $y\in Y^\hyp$, we have $f_{\varphi}^y = \,^T\!f_{\varphi}^y$. In particular, the map $f_{\varphi} : y \mapsto f_{\varphi}^y$ is piecewise $|k^\times|$-monomial and continuous on each interval inside~$Y^\hyp$.
\end{proposition}
\begin{proof}
Let~$y \in Y^{(2)}$. Let~$K$ be an algebraically closed complete valued extension of~$\sH(y)$. We use the notation of Theorem~\ref{thm:radializingskeleton} for~$y$ and~$\sigma_{K}(y)$.

The preimage of each element of~$\cD_{y} \setminus F_{y}$ being a disjoint union of elements of $\cD_{\sigma_{K}(y)} \setminus F_{\sigma_{K}(y)}$, there exists $D \in \cD_{y} \setminus F_{y}$ such that one of the connected components of~$\pi^{-1}_{K}(D)$ belongs to~$\cD_{\sigma_{K}(y)} \setminus F_{\sigma_{K}(y)}$. It follows that $^T\!f^{\sigma_{K}(y)} =\, ^T\!f^y$. 

Moreover, by Lemma~\ref{lem:fiber}, all the connected components of $\pi^{-1}_{K}(y) \setminus \{\sigma_{K}(y)\}$ are connected components of $Y_{K}\setminus\{\sigma_{K}(y)\}$. Since there are infinitely many of them, one of them belongs to $\cD_{\sigma_{K}(y)} \setminus F_{\sigma_{K}(y)}$ and it follows that $^T\!f^{\sigma_{K}(y)}=f^y$.

\medbreak

Let $y \in Y^\hyp$. Let~$I$ be an interval inside~$Y^\hyp$ containing~$y$ and let $(z_{n})_{n\ge 0}$ be a sequence of points of type~2 of~$I$ converging to~$y$. The sequence of Temkin's profile functions $(^T\!f^{z_{n}})_{n\ge 0}$ converges to Temkin's profile function~$^T\!f^{y}$. 

Let~$K$ be an algebraically closed complete valued extension of~$\sH(y)$. The map~$\sigma_{K}$ being continuous and sending intervals to intervals, as before, $^T\!f^{\sigma_{K}(z_{n})}$ tends to~$^T\!f^{\sigma_{K}(y)}$ when $n$ tends to infinity. Each~$z_{n}$ being of type~2, by the result above, we have $^T\!f^{\sigma_{K}(z_{n})} = {}^T\!f^{z_{n}} = f^{z_{n}}$. Since~$\sigma_{K}(y)$ is also of type~2, we have $^T\!f^{\sigma_{K}(y)} = f^{\sigma_{K}(y)} = f^{y}$ by Lemma~\ref{lem:profileK}. By uniqueness of the limit, we deduce that $f^{y} = {}^T\!f^{y}$.

The remaining part of the result now follows from the similar statement for Temkin's profile function. 
\end{proof}

\begin{remark}
The statement still holds, with the same proof, if~$k$ has positive characteristic. In particular, our method provides an alternative definition of the profile function at every point (see Notation~\ref{not:fyNy}) that is direct and makes no use of continuity.
\end{remark}

We now give examples of profile functions. As before, we work with a finite morphism $\varphi : Y \to X$ between quasi-smooth $k$-analytic curves.

\begin{example}\label{ex:tamelyramified}
Let $y\in Y^\hyp$. Assume that~$\varphi$ is residually separable at~$y$ (\ie the residual separable degree of~$\varphi$ at~$y$ is equal to the residual degree of~$\varphi$ at~$y$). Then, we have $f^y_{\varphi} = \mathrm{id}$.

Indeed, set $x := \varphi(y)$ and let $K$ be an algebraically closed complete valued extension of $\sH(y)$. By Lemma~\ref{lem:phiyK34} and Remark~\ref{rem:resdeghyp}, for each connected component~$D$ of~$\pi_{K}^{-1}(x)\setminus\{\sigma_{K}(x)\}$, the number of connected components of~$\varphi_{K}^{-1}(D) \cap \pi_{K}^{-1}(y)$ is equal to the degree of~$\varphi$ at~$y$, which is also the degree of~$\varphi_{K}$ at~$\sigma_{K}(y)$. It follows that the morphism induced by~$\varphi_{K}$ on each of these connected components has degree one, hence is an isomorphism.
\end{example}

\begin{example}\label{ex:Frobenius}
Suppose now that $\text{char}(\tilde{k})=p>0$ and let $a\in k^*$. Let $\vphi $ and $ \phi:\A^1_k\to \A^1_k$ be the Frobenius morphism and the off-centered Frobenius morphism at $a$, respectively. Their action on rational points is given by $\vphi:x \in \A^1_{k}(k) \mapsto x^p \in \A^1_{k}(k)$ and $\phi:x\in \A^1_k(k)\mapsto (x+a)^p-a^p\in \A^1_k(k)$. Then, a direct computation shows that, for each $\rho>0$ and each $s \in [0,1]$, we have 
\[
f_{\varphi}^{\eta_{\rho}}(s) = 
\begin{cases}
|p|s \quad \textrm{ if } s\in [0,|p|^{\frac{1}{p-1}})\\
s^p \quad \quad \textrm{if } s\in [|p|^{\frac{1}{p-1}},1].\\
\end{cases}
\] 
On the other side, for $\rho\in (0,|a|\, |p|^{\frac{1}{p-1}}]$ we have
\[
f_{\phi}^{\eta_{\rho}}(s) = 
1, \quad \textrm{ for every } s\in [0,1],
\]
while for $\rho\in (|a|\, |p|^{\frac{1}{p-1}},|a|]$ we have
\[
 f_{\phi}^{\eta_{\rho}}(s) = 
\begin{cases}
|p|\, |a|^{p-1}\rho^{1-p}s \quad \textrm{ if } s\in [0,|p|^{\frac{1}{p-1}}|a|\rho^{-1})\\
s^p \quad \quad \quad \quad \quad \quad \, \, \textrm{if } s\in [|p|^{\frac{1}{p-1}}|a|\rho^{-1},1]\\
\end{cases}
\]
and for $\rho\geq |a|$ we have $f_{\phi}^{\eta_\rho}=f_{\vphi}^{\eta_\rho}$.

\end{example}

The last examples can be widely generalized. In order to do so, we will need to use the notion of different~$\delta_{L/K}$ of an extension of one-dimensional analytic fields~$L/K$ from~\cite{CTT14}. We refer to this article for the definition and basic properties and simply give a way to compute it in the case we are interested in. 

Let~$y \in Y^\hyp$. Recall that a tame parameter at~$y$ is an element $t_{y} \in \cO_{y}$ such that the extension of valued fields $\sH(y)/\widehat{k(t_{y})}$, where $\widehat{k(t_{y})}$ denotes the closure of~$k(t_{y})$ in~$\sH(y)$, is tamely ramified. The element~$t_{y}$ induces a morphism $\psi_{y}$ from some affinoid domain of~$Y$ containing~$y$ to~$\A^1_{k}$. We denote by $r(t_{y})$ the radius of the point~$\psi_{y}(y)$.

For each $y\in Y^\hyp$, if~$t_{y}$ is a tame parameter at~$y$ and~$t_{x}$ a tame parameter at~$x:=\varphi(y)$, by \cite[Corollary~2.4.6]{CTT14}, we have
\[\delta_{\varphi}(y) := \delta_{\sH(y)/\sH(x)} = \left| \frac{dt_{x}}{dt_{y}} \right| \frac{r(t_{y})}{r(t_{x})}.\]

\begin{lemma}\label{lem:differentK}
Let~$y\in Y^\hyp$. For each complete valued extension~$K$ of~$k$, we have
\[\delta_{\varphi_{K}}(\sigma_{K}(y)) = \delta_{\varphi}(y).\]
\end{lemma}
\begin{proof}
Let~$K$ be a complete valued extension of~$k$. Assume that~$y$ is of type~3 or~4. Then, $y$ has a neighborhood isomorphic to an analytic domain of the affine line, hence there exists $t_{y} \in \cO_{y}$ such that $\widehat{k(t_{y})}$ is isomorphic to~$\sH(y)$. Then $t_{y}$ is a parameter at~$y$ and its image~$t_{\sigma_{K}(y)}$ in $\sH(\sigma_{K}(y))$ is also a parameter at~$\sigma_{K}(y)$. Note also that the explicit description of the points of the line shows that we have $r(t_{\sigma_{K}(y)}) = r(t_{y})$. Arguing similarly for the point~$x$ (which has the same type as~$y$), the formula follows.

Let us now assume that~$y$ is of type~2. Choose a tame parameter~$t_{y}$ at~$y$. It follows from Lemma~\ref{lem:phiyK} that its image~$t_{\sigma_{K}(y)}$ in $\sH(\sigma_{K}(y))$ is also a parameter at~$\sigma_{K}(y)$. By \cite[Lemma~2.9]{Poi-Pul2}, we have $\sigma_{K}(\psi_{y}(y)) = \psi_{\sigma_{K}(y)}(\sigma_{K}(y))$, hence $r(t_{\sigma_{K}(y)}) = r(t_{y})$ by a computation on the line as before. Arguing similarly for~$x$, the result follows.
\end{proof}

\begin{example}\label{ex:degreep}
Let~$y\in Y^\hyp$. Suppose now that $\text{char}(\tilde{k})=p>0$ and that $\vphi$ is residually purely inseparable of degree $p$ at $y$ (\ie the residual separable degree of~$\varphi$ at~$y$ is equal to the residual degree of~$\varphi$ at~$y$ and equal to~$p$). Then, for each $s \in [0,1]$, we have 
\[
f^y_{\varphi}(s) = 
\begin{cases}
\delta_{\varphi}(y)s \quad \, \textrm{ if } s\in [0,\delta_{\varphi}(y)^{\frac{1}{p-1}})\\
s^p \qquad \quad \textrm{ if } s\in [\delta_{\varphi}(y)^{\frac{1}{p-1}},1].\\
\end{cases}
\]

Indeed, let~$K$ be an algebraically closed complete valued extension of~$\sH(y)$. By Lemma~\ref{lem:profileK}, we have $f_{\varphi_{K}}^{\sigma_{K}(y)} = f_{\varphi}^y$ and by Lemma~\ref{lem:differentK}, we have $\delta_{\varphi_{K}}(\sigma_{K}(y)) = \delta_{\varphi}(y)$, hence, up to replacing~$y$ by~$\sigma_{K}(y)$, we may assume that~$y$ is of type~2. The result then follows from Proposition~\ref{prop:comparison} and from the similar result for Temkin's profile function, see \cite[Theorem 3.3.7]{TemkinHerbrand}.
\end{example}

\section{$p$-adic differential equations}

\subsection{Modules with an integrable connection} Let $X$ be a quasi-smooth $k$-analytic curve. By a differential equation on $X$ we mean a pair $(\sE,\nabla)$ where $\sE$ is a locally free $\cO_X$-module $\sE$ of finite type equipped with an integrable connection, \ie a $k$-linear map $\nabla:\sE\to  \Omega^1_{X}\otimes \sE$ which satisfies the Leibniz rule
\[
\nabla(f\otimes e)=df\otimes e +f\nabla(e).
\]
We say that the differential equation has rank~$r$ if~$\sE$ has rank~$r$. We put $\sE^{\nabla}:=\ker \nabla$ and the elements of the latter sheaf are called the horizontal sections of $(\sE,\nabla)$. 

If $K/k$ is an extension of valued fields, then one may consider the pullback $\pi^\ast_K(\sE,\nabla)$, which is a differential equation on the curve $X_K$.

\subsection{Multiradius of convergence}
We first define radii of convergence at rational points of discs. 

%Let~$D$ be a strict open disc. Let $x\in D(k)$. Let us identify~$D$ with the open unit disc \textit{via} a coordinate $T_x:D\iso D(0,1^-)$ which sends $x$ to $0$. For each $s\in (0,1)$, we set $D(x,s^{\pm}):=T_x^{-1}\big(D(0,s^{\pm})\big)$. 

\begin{definition}\label{def:multiradius}
Let $(\sE,\nabla)$ be a differential equation of rank~$r$ on a strict open disc~$D$. Let $x\in D(k)$. Let $i\in\{1,\dotsc,r\}$. We define the $i^\textrm{th}$ radius of convergence of~$\ena$ at~$x$ to be  
\begin{equation}\label{radius}
\cR_i\big(x,(\sE,\nabla)\big):=\sup\big\{s\in (0,1) \mid \dim_{k}\big(H^0(D(x,s^-),(\sE,\nabla))\big) \ge r- i+1 \big\}
\end{equation}
and the multiradius of convergence of~$\ena$ at~$x$ to be 
\begin{equation*}\label{multiradius}
\cM\cR\big(x,(\sE,\nabla)\big) := \big(\cR_1(x,(\sE,\nabla)),\dots,\cR_r(x,(\sE,\nabla))\big).
\end{equation*}
\end{definition}

We now define radii of convergence at a non-rational point~$x$ of a quasi-smooth $k$-analytic curve~$X$. Let~$K$ be a complete algebraically closed extension of~$\sH(x)$. By Lemma~\ref{lem:fiber}, the connected components of $\pi_{K}^{-1}(x) \setminus \{\sigma_{K}(x)\}$ are strict open discs. Let $x_{K}$ be a $K$-rational point above $x$ and let ~$D^g_{x}$ be the connected component of $\pi_{K}^{-1}(x) \setminus \{\sigma_{K}(x)\}$ that contains $x_K$.

\begin{notation}\label{not:dimH0}
For $s\in (0,1)$, we set 
\[\dim_{k}\big(H^0(D(x,s^-),(\sE,\nabla))\big) := \dim_{K}\big(H^0(D(x_{K},s^-),\pi_{K}^*(\sE,\nabla)_{|D_{x}^g})\big).\]
\end{notation}

\begin{definition}\label{def:multiradiusarbitrary}
Let $(\sE,\nabla)$ be a differential equation of rank~$r$ on $X$. Let $i\in\{1,\dotsc,r\}$. We define the $i^\textrm{th}$ radius of convergence of~$\ena$ at~$x$ to be  
\begin{equation*}
\cR_i\big(x,(\sE,\nabla)\big):=\sup\big\{s\in (0,1) \mid \dim_{k}\big(H^0(D(x,s^-),(\sE,\nabla))\big) \ge r- i+1 \big\}
\end{equation*}
and the multiradius of convergence of~$\ena$ at~$x$ to be 
\begin{equation*}
\cM\cR\big(x,(\sE,\nabla)\big) := \big(\cR_1(x,(\sE,\nabla)),\dots,\cR_r(x,(\sE,\nabla))\big).
\end{equation*}
\end{definition}

%\begin{definition}\label{def:multiradiusarbitrary}
%Let $(\sE,\nabla)$ be a differential equation of rank~$r$ on $X$. Let $i\in\{1,\dotsc,r\}$. We define the $i^\textrm{th}$ radius of convergence of~$\ena$ at~$x$ to be  
%\begin{equation*}
%\cR_i\big(x,(\sE,\nabla)\big):= \cR_i\big(x_{K},\pi_{K}^*(\sE,\nabla)_{|D_{x}^g}\big)
%\end{equation*} 
%and the multiradius of convergence of~$\ena$ at~$x$ to be 
%\begin{equation*}
%\cM\cR\big(x,(\sE,\nabla)\big) := \big(\cR_1(x,(\sE,\nabla)),\dots,\cR_r(x,(\sE,\nabla))\big).
%\end{equation*}
%\end{definition}

The notation and definition above are independent of the choice of the field~$K$ and of the $K$-rational point $x_K$ above $x$ (see \cite[Section~2.3]{Poi-Pul2}). 

%\J{Maybe we could define $ \dim_{k}\big(H^0(D(x,s^-),(\sE,\nabla))\big)$ for non-rational points here (and use it later).}

\begin{remark}
Our definition of the (multi)radius of convergence corresponds to what is usually called the spectral or intrinsic multiradius of convergence in the literature. 
\end{remark}

\subsection{Convergence polygons}

Let $(\sE,\nabla)$ be a differential equation of rank~$r$ on a quasi-smooth $k$-analytic curve $X$.
 
\begin{defn}\label{conv polygon}
 Let $x \in X^\hyp$. For each $i=1,\dotsc,r$, we set 
 \[h_i\big(x,\ena\big):= - \sum_{j=1}^i \log(R_j\big(x,\ena\big)).\]
The {\em convergence polygon} $\cN\big(x,\ena\big)$ of $(\sE,\nabla)$ at the point $x$ is the concave polygon with vertices 
 \[
 (0,0), \quad\big(i,h_i(x,\ena)\big),\quad i=1,\dots, r.
 \] 
The {\em height} of the convergence polygon at the point $x$ is $h\big(x,\ena\big):=h_r\big(x,\ena\big)$. Again, if no confusion can arise, we may omit writing $(\sE,\nabla)$.
\end{defn}

 \begin{theorem}\label{radius continuity} 
 Let $i\in \{1,\dotsc,r\}$. Then, the $i^\textrm{th}$ radius function 
 \[\cR_{i}\big(\cdot,\ena\big): X^\hyp\to (0,1]^r\] 
 is continuous and piecewise $|k^*|$-monomial on each interval in~$X^\hyp$. Moreover, for each $x\in X^{(2)}$, it is monomial of degree~1 along almost every direction emanating from~$x$.
\end{theorem}
\begin{proof}
For an interval corresponding to a skeleton of an annulus, this follows from \cite[Theorem 11.3.2]{pde}. However, instances of it were known much earlier and we refer to the Notes at the end of Section~11 in \textit{ibid.} for a detailed account. For higher genus curves and intervals with type 2 and 3 endpoints and $\cR_1$ this is in \cite{Bal10}. The rest of the cases are in \cite[Theorem~3.6]{Poi-Pul2}.   
\end{proof}

Let~$x$ be a non-rational point of~$X$, $\vt$ be a tangent direction at~$x$ and~$I$ be an interval in $X$ corresponding to $\vt$. It follows from the previous theorem that the function $h(x',\ena)$ has constant degree for $x'\in I$ close enough to $x$. Denote this degree by~$n$.
\begin{defn}
The quantity $-n$ is called the irregularity of $\ena$ at the point $x$ in the direction $\vt$ and is denoted by $\Irr_\vt\big(x,\ena\big)$. 
If $x$ is clear from the context, we may also write $\Irr_\vt\ena$.
\end{defn}

In explicit terms let $A_\vt$ be a small enough strict open annulus in $X$ attached to $x$ that corresponds to $\vt$ and normalized via some coordinate $T:A_\vt\iso A(0;q,1)$ so that the point~$\eta_{\rho}$ of the skeleton tends to~$x$ when~$\rho$ tends to~1.
%for the points on the skeleton $\eta_\rho$ we have $\eta_\rho\rightarrow x$ for $\rho\rightarrow 1$. 
Let us define the function $I \colon s \in (\log r,0)\mapsto h(\eta_{\rho}) \in \R$,  where $s=\log\rho$. We then have
%$I:(\log r,0)\to \R$ by $s\mapsto I(s)=h(\eta_{\rho})$, where $s=\log\rho$.
 \[
 \Irr_\vt\big(x,\ena\big):=-\frac{d}{ds}I(0-).
 \]
  
It follows from Theorem \ref{radius continuity}, and more precisely from piecewise monomiality of the multiradius function, that the irregularity is well defined.

The previous definition can be globalized in the following way. Let $\Gamma$ be a finite subset of $T_xX$. Then, we define the Laplacian of $\ena$ at $x$ along $\Gamma$ to be  
\[
\Delta_x\big(\Gamma,\ena\big):=\sum_{\vt\in \Gamma}\Irr_\vt\ena.
\]

\subsection{Pushforward of differential equations}
Let $\vphi:Y\to X$ be a finite \'etale morphism of quasi-smooth $k$-analytic curves. Let $(\sE,\nabla_{\sE})$ be a differential equation on $X$.

Set $\sF:=\vphi_*(\sE)$. Since $\vphi$ is \'etale, we have $\Omega^1_Y=\vphi^*\Omega^1_X$ hence, due to the projection formula, a map
\[\nabla_\sF:\vphi_*(\sE)\to  \vphi_*(\Omega^1_Y\otimes_{\sO_Y}\sE) \simeq \Omega^1_X\otimes_{\sO_X}\vphi_*(\sE).\]
The couple $\vphi_*(\sE,\nabla_{\sE})=:(\sF,\nabla_{\sF})$ is a $p$-adic differential equation on $Y$ called the {\em pushforward} of $(\sE,\nabla_{\sE})$ by~$\varphi$. 

To lighten the notation, we will use the same symbol $\nabla$ for both connections $\nabla_{\sE}$ and $\nabla_{\sF}$, hoping that it will be clear from the context what connection it represents.
\begin{rmk}\label{rmk:pushforward}
 Note that it follows from the definition above that for any analytic domain~$U$ of~$X$, we have
 \begin{equation}\label{solutions pf}
 H^0\big(\vphi^{-1}(U),\ena\big)\xrightarrow[]{\sim} H^0\big(U,\fna\big).
 \end{equation}
 This fact will be extensively used in the next section.
\end{rmk}

\section{The main results}

\subsection{Pushforward of a connection: open discs} 

Let $\vphi:D_2\to D_1$ be a finite \'etale morphism of degree $d$ between srict open discs. Let $\ena$ be a differential equation of rank~$r$ on~$D_2$ and let $\fna:=\vphi_*\ena$ be its pushforward on~$D_2$ by~$\vphi$.

\begin{defn}
Let $x\in D_1(k)$ be a rational point. For $s\in (0,1]$, we set
\begin{equation*}
\Phi_x\big(s,\vphi,\ena\big) := \dim_k H^0\big(\varphi^{-1}(D(x,s^-)),\ena\big).
\end{equation*}
%For an arbitrary point $x\in D_{1}$, we set  
%\begin{equation*}
%\Phi_x\big(s,\vphi,\ena\big) := \Phi_{x_{K}}\big(s,\vphi_{K},\pi^*_{K}\ena\big),
%\end{equation*}
%where~$K$ is a complete algebraically closed extension of~$\sH(x)$.
% is the completion of an algebraic closure of the complete residue field~$\sH(x)$.
\end{defn}

It is clear that, for each $x\in D_{1}(k)$,
%$x\in D_1$, 
the function $\Phi_x$ is non-increasing, left-continuous and bounded by~$rd$. In particular, it has finitely many breaks. Note also that, since, by Lutz' theorem \cite[Th\'eor\`eme IV]{Lut37}, differential equations admit full sets of solutions in the neighborhood of rational points, we have $\Phi_x(s)=rd$ for~$s$ close enough to~0.

\begin{remark}\label{rem:PhiOd}
Assume that $\ena=(\Os^{r},d^{r})$ is the constant connection. In this case, for each $x\in D_1(k)$ and every $s\in (0,1]$, $\dim_k H^0\big(\varphi^{-1}(D(x,s^-)),(\Os^{r},d^{r})\big)$ is nothing but the number of connected components of $\varphi^{-1}\big(D(x,s^-)\big)$ multiplied by the rank~$r$. We deduce that, for each $x\in D_{1}$, we have $\Phi_{x} = rN_{x}$.
\end{remark}

For $R \in [0,1)$, we set $\Phi_{x}(R^+) := \lim_{s \to R^+} \Phi_{x}$.

\begin{proposition}\label{pushforward over discs}
Let $x\in D_1(k)$. 
%Let $x\in D_1$. 
We have
\begin{equation*}
\#\{i \in \{1,\dotsc,rd\} \mid \cR_{i}(x,\fna) = 1\} = \Phi_{x}(1)
\end{equation*}
and, for each $R\in (0,1)$,
\begin{equation*}
\#\{i \in \{1,\dotsc,rd\} \mid \cR_{i}(x,\fna) = R\} = \Phi_{x}(R) - \Phi_{x}(R^+).
\end{equation*}
In particular, the radii of convergence of~$\fna$ at the point~$x$ that are smaller than~1 are exactly the break-points of~$\Phi_{x}$.
\end{proposition}
\begin{proof}
%Due to the definitions, we may assume that~$x$ is a $k$-rational point. 
Note that~$\sF$ has rank~$rd$. 
Following formula \eqref{radius} and using \eqref{solutions pf}, we may write
 \begin{align*}
 \cR_{i}\big(x,\fna\big)&=\sup\Big\{s\in(0,1)\mid \dim_kH^0\big(D(x,s^-),\fna\big)\geq rd-i+1\Big\}\\
 &=\sup\Big\{s\in(0,1)\mid \dim_kH^0\big(\varphi^{-1}(D(x,s^-)),\ena\big)\geq rd-i+1\Big\}\\
 &=\sup\big\{s\in(0,1)\mid \Phi_x(s)\geq rd-i+1\big\}.
 \end{align*}
The result follows.
\end{proof}

In the rest of the section, we will study manifestations of the previous theorem for particular cases of the connection $\ena$.

\begin{corollary}\label{push constant connection}
Assume that $\ena=(\Os,d)$ is the constant connection on $D_2$ and let $x\in D_{1}(k)$.
%$x\in D_1$. 
Let $0 = s_{0} <s_1<\dots<s_{n-1} < s_{n}=1$ be the breaks of the function $N_x$. For $i=1,\dots,n$, set $n_{x,i}:=N_x(s_i)$. Then, the radii of convergence of $\vphi_{\ast}(\Os,d)$ at~$x$ are 

\begin{equation}\begin{cases}\label{pushforward constant formula}
\cR_{1}\big(x,\vphi_{\ast}(\Os,d)\big) = \dotsb = \cR_{d-n_{x,2}}\big(x,\vphi_{\ast}(\Os,d)\big) = s_{1}~;\\
\cR_{d-n_{x,2}+1}\big(x,\vphi_{\ast}(\Os,d)\big) = \dotsb = \cR_{d-n_{x,3}}\big(x,\vphi_{\ast}(\Os,d)\big) = s_{2}~;\\
\vdots\\
\cR_{d-n_{x,n-1}+1}\big(x,\vphi_{\ast}(\Os,d)\big) = \dotsb = \cR_{d-1}\big(x,\vphi_{\ast}(\Os,d)\big) = s_{n-1}~;\\
\cR_{d}\big(x,\vphi_{\ast}(\Os,d)\big) = 1.
\end{cases}
\end{equation}
\end{corollary}
\begin{proof}
The result follows directly from Proposition~\ref{pushforward over discs} and Remark~\ref{rem:PhiOd}. Indeed, we have 
\[\Phi_{x}(1) = N_{x}(1) = 1\] 
since it is, by definition, the number of connected components of $\varphi^{-1}(D(x,1^-)) = \varphi^{-1}(D_{1}) = D_{2}$ and, for each $i=1,\dotsc,n-1$, we have 
\[\Phi_x(s_{i}) - \Phi_x(s_{i}^+) = \Phi_{x}(s_{i}) - \Phi_{x}(s_{i+1}) = n_{x,i} - n_{x,i+1},\]
with $n_{x,n} = N_{x}(1) = 1$ and $n_{x,1} = N_{x}(s_{1}) = d$. 
\end{proof}
 
For the next corollary and the results that follow we will use the following operation. If~$F$ and~$F'$ are two finite non-decreasing families of real numbers, we denote by~$F \ast F'$ the family obtained by concatenation and reordering. If~$n$ is a positive integer, we denote by~$F^{\ast n}$ the family $F \ast \dotsb \ast F$ with $n$~copies of~$F$. 
 
\begin{corollary}\label{cor:radial over discs}
Assume that $\vphi:D_2\to D_1$ is radial with profile function $f:=f_{\varphi}$. Let $x\in D_{1}(k)$ and let $0 = s_{0} <s_1<\dots<s_{n-1}<s_n=1$ be the breaks of the function $N_x$. For $i=1,\dots,n$, set $n_{x,i}:=N_x(s_i)$. Suppose that there exist $R_{1} \le \dotsb \le R_{r}$ such that, for each $y\in \vphi^{-1}(x)$, we have 
\[\cM\cR\big(y,\ena\big)=(R_1,\dots,R_r).\] 
For $i=1,\dots,r$, let $m_{i}$ be the unique element of $\{0,\dotsc,n-1\}$ such that $s_{m_{i}}<f(R_i)\leq s_{m_{i}+1}$ and set
\[F_i:=\big(s_1,\dots,s_1,\dots,s_{m_{i}}\dots,s_{m_{i}},f(R_i)\dots,f(R_i)\big),\] 
where, for $j=1,\dotsc,m_{i}$, $s_j$ appears $n_{x,j}-n_{x,j+1}$ times and $f(R_i)$ appears $n_{x,m_{i}+1}$ times. Then, we have
\begin{equation}\label{eq:astfirst}
\cM\cR(x,\fna)=\bigast_{i=1}^rF_i.
\end{equation}
\end{corollary}
\begin{proof}
By assumption, for each $y\in \varphi^{-1}(x)$ and each $t \in (0,1]$, we have
\[\dim_{k} H^0 (D(y,t^-),\ena) =  \#\{i \in \{1,\dotsc,r\} \mid R_{i} \ge t\}.\]

Let $s\in (0,1]$. We denote by $D_1(s),\dotsc,D_{N_x(s)}(s)$ the open discs that are the connected components of $\varphi^{-1}(D(x,s^-))$. Since~$\varphi$ is radial, they all have the same radius, namely~$f^{-1}(s)$. We deduce that
\begin{align*}
\Phi_{x}(s) &= \dim_{k} H^0\big(\varphi^{-1}(D(x,s^-)),\ena\big)\\
 &= \sum_{j = 1}^{N_{x}(s)} \dim_{k} H^0\big(D_{j}(s),\ena\big)\\
 &= N_{x}(s) \cdot \#\{i \in \{1,\dotsc,r\} \mid R_{i} \ge f^{-1}(s)\}.
\end{align*}
If follows that the set of break-points~$B$ of the function~$\Phi_{x}$ is contained in
\[\{s_{1},\dotsc,s_{n},f(R_{1}),\dotsc,f(R_{r})\}.\]
Noting that $\#\{i \in \{1,\dotsc,r\} \mid R_{i} \ge f^{-1}(s)\} =0$ if, and only if, $s> f(R_{r})$, it follows that
\[B = \{s_{1},\dotsc,s_{m_{r}},f(R_{1}),\dotsc,f(R_{r})\}.\]

Let~$b\in B$. We will now compute the multiplicity of~$b$ in $\cM\cR(x,\fna)$ using Proposition~\ref{pushforward over discs} and show that it is equal to its multiplicity in $F := \bigast_{i=1}^r F_{i}$. To this end, remark that, for each $b\in (0,1)$, we have
\[\Phi_{x}(b^+) = N_{x}(b^+) \cdot \#\{i \in \{1,\dotsc,r\} \mid R_{i} > f^{-1}(b)\},\]
where $N_{x}(b^+) = \lim_{s \to b^+} N_{x}(s)$.

$\bullet$ Assume that there exists $u\in \{1,\dotsc,m_{r}\}$ such that $b = s_{u}$ and $b \ne f(R_{v})$ for each $v\in \{1,\dotsc,r\}$.

Then, we have $b=s_{u}\le s_{m_{r}} < f(R_{r}) \le 1$, hence~$b$ appears in~$\cM\cR(x,\fna)$ with multiplicity
\[\Phi_{x}(b) - \Phi_{x}(b^+) = (n_{x,u} - n_{x,u+1})  \cdot \#\{i \in \{1,\dotsc,r\} \mid f(R_{i}) \ge s_{u}\}.\]
On the other hand, for each $i\in \{1,\dotsc,r\}$, $s_{u}$ appears in~$F_{i}$ with multiplicity $n_{x,u} - n_{x,u+1}$ when $s_{u} < f(R_{i})$ (which is equivalent to $s_{u} \le f(R_{i})$ in this case) and~0 otherwise. The result follows.

$\bullet$ Assume that there exists $v\in \{1,\dotsc,r\}$ such that $b = f(R_{v}) < s_{m_{v}+1}$.

Then, we have $b<1$ and~$b$ appears in~$\cM\cR(x,\fna)$ with multiplicity
\[\Phi_{x}(b) - \Phi_{x}(b^+) = n_{x,m_{v}+1} \cdot \#\{i \in \{1,\dotsc,r\} \mid R_{i} = R_{v}\}.\]
On the other hand, for each $i\in \{1,\dotsc,r\}$, $b$ appears in~$F_{i}$ with multiplicity~$n_{x,m_{v}+1}$ if $R_{i}=R_{v}$ and~0 otherwise. The result follows.

$\bullet$ Assume that there exists $v\in \{1,\dotsc,r\}$ such that $b = f(R_{v}) = s_{m_{v}+1} <1$.

Then, $m_{v} \le n-2$ and $b$ appears in~$\cM\cR(x,\fna)$ with multiplicity
\begin{eqnarray*}
\Phi_{x}(b) - \Phi_{x}(b^+) &=& n_{x,m_{v}+1} \cdot \#\{i \in \{1,\dotsc,r\} \mid R_{i} \ge R_{v}\}\\
&& - n_{x,m_{v}+2} \cdot \#\{i \in \{1,\dotsc,r\} \mid R_{i} > R_{v}\}\\
&=& (n_{x,m_{v}+1} - n_{x,m_{v}+2}) \cdot \#\{i \in \{1,\dotsc,r\} \mid R_{i} > R_{v}\}\\
&&+ n_{x,m_{v}+1} \cdot \#\{i \in \{1,\dotsc,r\} \mid R_{i} = R_{v}\}.
\end{eqnarray*}

Let $i\in \{1,\dotsc,r\}$. If $R_{i} > R_{v}$, then $b=s_{m_{v}+1}$ appears in~$F_{i}$ with multiplicity $n_{x,m_{v}+1} - n_{x,m_{v}+2}$. If $R_{i} = R_{v}$, then $b=f(R_{v})$ appears in~$F_{i}$ with multiplicity $n_{x,m_{v}+1}$. If $R_{i} < R_{v}$, then $b$ does not appear in~$F_{i}$. The result follows.

$\bullet$ Assume that $b=1$.

Then we have $R_{r} = 1$, $m_{r} = n-1$ and~$b$ appears in~$\cM\cR(x,\fna)$ with multiplicity
\[\Phi_{x}(1) = n_{x,n} \cdot \#\{i \in \{1,\dotsc,r\} \mid R_{i} = 1\}.\]  
On the other hand, for each $i\in \{1,\dotsc,r\}$, $b=1$ appears in~$F_{i}$ with multiplicity~$n_{x,n}$ if $R_{i}=1$ and~0 otherwise. The result follows.
\end{proof}

\subsection{Pushforward of a connection: the general case}

Let $\varphi \colon Y \to X$ be a finite \'etale morphism between quasi-smooth $k$-analytic curves. Let $\ena$ be a differential equation of rank~$r$ on~$Y$ and let $\fna:=\vphi_*\ena$ be its pushforward on~$X$ by~$\vphi$. Let~$y\in Y^\hyp$ and set $x:=\vphi(y)$.

We will use Notation~\ref{not:fyNy}. Let $0 = s^y_{0} < s^y_{1} < \dotsb < s^y_{n(y)-1} < s^y_{n(y)} = 1$ be the break-points of the function $(f^{y})^{-1}$ (or $N^{y}$). For $i=1,\dotsc,n(y)$, set $n^y_{i} := N^{y}(s^y_{i})$.

Let $R\in (0,1]$. Let $m \in \{0,\dotsc,n(y)-1\}$ be such that $s^y_{m} < f^{y}(R) \le s^y_{m+1}$. Denote by $F^{y}(R)$ the family
\[F^y(R) = (s^y_{1},\dotsc, s^y_{1}, \dotsc, s^y_{m},\dotsc,s^y_{m},f^{y}(R),\dotsc,f^{y}(R)),\]
where, for each $i=1,\dotsc,m$, $s^y_{i}$ appears $n^y_{i} - n^y_{i+1}$ times and $f^{y}(R)$ appears $n^y_{m+1}$ times.

\begin{theorem}\label{push general} 
The multiradius of convergence of $\fna = \varphi_{\ast}\ena$ at~$x$ is given by 
\[
\cM\cR(x,\fna) = \bigast_{\substack{y\in \varphi^{-1}(x) \\ 1\le i\le r}} F^y\Big(\cR_{i}\big(y,\ena\big)\Big)^{\ast \ks^{y}}.
\]
\end{theorem}
\begin{proof}
Denote by~$y_{1},\dotsc,y_{q}$ the preimages of~$x$ by~$\varphi$. Since~$\varphi$ is finite, there exists an open neighborhood~$U$ of~$x$ in~$X$ and, for each $j=1,\dotsc,q$, an open neighborhood~$V_{j}$ of~$y_{j}$ in~$Y$ such that 
\[\varphi^{-1}(U) = \bigsqcup_{j=1}^q V_{j}.\]
For each $j=1,\dotsc,q$, denote by $\varphi_{j} \colon V_{j} \to U$ the morphism induced by~$\varphi$.

Since the radii at a point only depend of the germ of curve at this point, we may assume that~$X=U$. We deduce that
\[\cM\cR(x,\fna) = \bigast_{j=1}^q \cM\cR(x,(\varphi_{j})_{*}(\sE, \nabla)_{|V_{j}}).\]
As a consequence, we may assume that $\varphi^{-1}(x) = \{y\}$.

Let~$K$ be an algebraically closed complete valued extension of~$\Hs(x)$. Let~$x'$ be a $K$-rational point of~$X_{K}$ over~$x$. Let~$D'$ be the connected component of $\pi^{-1}_{K}(x)\setminus\{\sigma_{K}(x)\}$ containing~$x'$. By definition, we have $\cM\cR(x,\fna) = \cM\cR(x',\pi_{K}^*\fna_{|D'})$.

Let $D'_{1},\dotsc,D'_{s}$ be the connected components of~$\varphi_{K}^{-1}(D')$. For each $j=1,\dotsc,s$, let $\psi_{j} \colon D'_{j} \to D'$ be the morphism induced by~$\varphi_{K}$ and denote its degree by~$d_{j}$. We have
\[\cM\cR(x',(\varphi_{K})_{*}\pi_{K}^*\ena_{|D'}) = \bigast_{j=1}^s \cM\cR(x',(\psi_{j})_{*}\pi_{K}^*\ena_{|D'_{j}}).\]

By Lemma~\ref{lem:fiber}, the $D'_{j}$'s are strict open discs, the morphisms~$\psi_{j}$'s are radial and all the profile functions coincide, the common function being~$f^{y}$.

Let us set $(R_1,\dots,R_r):=\cM\cR^{\sp}\big(y,\ena\big)$. Then, for every rational point $y'$ in $D'_j$ we have $\cM\cR(y',\pi_{K}^*\ena_{|D'_{j}})=(R_1,\dots,R_r)$. 
By further noting that $s=\ks^{y}$ and using Corollary~\ref{cor:radial over discs}, we find 
\[\cM\cR(x',(\psi_{j})_{*}\pi^*\ena_{|D'_{j}}) = \bigast_{i=1}^rF^{y}(R_i),\]
hence 
\[\cM\cR(x,\fna) = \bigast_{i=1}^rF^{y}(R_i)^{\ast \ks^{y}}.\]
\end{proof}

\begin{remark}
It is possible to give a proof of Theorem~\ref{push general} that does not use Corollary~\ref{cor:radial over discs} but the simpler Corollary~\ref{push constant connection} about the constant connection if one is willing to decompose the differential module $\ena$ with respect to the radii (which is quite an involved result, see~\cite[Corollary~3.6.9]{Poi-Pul3} or~\cite[Proposition~2.2.9]{Kedlayalocalglobal}). Indeed, with the notation of the proof above, the decomposition theorem allows to reduce to the case where $(R_{1},\dotsc,R_{r}) = (R,\dotsc,R)$. Then, there are no radii at~$x$ bigger than~$f^y(R)$ and the radii smaller than~$f^y(R)$ are the same as those of the pushforward of the constant connection of rank~$r$. This also gives a conceptual explanation of the reason why the family of radii at~$x$ may be split into families associated to the different radii at~$y$.
\end{remark}

As a direct consequence, we can compute the radii of convergence in the situation of Examples \ref{ex:tamelyramified}, \ref{ex:Frobenius} and \ref{ex:degreep} by using the simple descriptions of the corresponding profile functions.

\begin{corollary}\label{cor:specialcases}
 \begin{enumerate}[(i)]
  \item Suppose that $\vphi$ is residually separable of degree~$d$ at~$y$ (\cf Example \ref{ex:tamelyramified}). Then,  we have
  \[
 \cR_{id-j}(x,\fna)=\cR_{i}(y,\ena),\quad i=1,\dots, r,\quad j=0,\dots,d-1.
 \]
 
 \item Suppose that $p := \text{char}(\tilde{k})>0$ and that~$\varphi$ is the Frobenius morphism on a annulus $\vphi: x \in A(0;r_{1},r_{2})\mapsto x^p \in A(0;r_{1}^p,r_{2})$. Let $\rho \in (r_{1},r_{2})$. In this case, we have $n(\eta_{\rho})=2$, $s^1_{\eta_{\rho}} = |p|^{- \frac{p}{p-1}}$, $n_{1}^{\eta_{\rho}}=p$ and $n_{2}^{\eta_{\rho}}=1$. Set $i_0:=\max \{i \in \{1,\dots, r\} \mid \cR_{i}(\eta_{\rho},\ena)\leq |p|^{-\frac{p}{p-1}}\}$. Then we have
   \begin{align*}
 &\cR_{ip-j}(\eta_{\rho^d},\fna)= |p|^{-1}\,\cR_{i}(\eta_{\rho},\ena),&\quad &i=1,\dots, i_0, \quad j=0,\dots,p-1;\\
 &\cR_{i}(\eta_{\rho^d},\fna)= |p|\cdot|p|^{-\frac{1}{p-1}}=|p|^{-\frac{p}{p-1}},&\quad & i=i_0p+1,\dots,(p-1)r+i_{0};\\
 &\cR_{(p-1)r+i}(\eta_{\rho^d},\fna)=\cR_{i}(\eta_{\rho},\ena)^p, &\quad & i=i_0+1,\dots,r.
 \end{align*}

 \item Let $y\in Y^\hyp$. Suppose that $p := \text{char}(\tilde{k}) >0$ and that~$\vphi$ is residually purely inseparable of degree~$p$ at~$y$. Let $\delta := \delta_{\varphi}(y)$ be the different of~$\varphi$ at~$y$ (\cf Example \ref{ex:degreep}). Set $i_0:=\max \{i \in \{1,\dots, r\} \mid \cR_{i}(y,\ena)\leq \delta^{\frac{p}{p-1}}\}$. Then, we have
  \begin{align*}
 &\cR_{ip-j}(x,\fna)=\delta\,\cR_{i}(y,\ena),&\quad &i=1,\dots, i_0, \quad j=0,\dots,p-1;\\
 &\cR_{i}(x,\fna)= \delta\cdot\delta^{\frac{1}{p-1}}=\delta^{\frac{p}{p-1}},&\quad & i=i_0p+1,\dots,(p-1)r+i_0;\\
 &\cR_{(p-1)r+i}(x,\fna)=\cR_{i}(y,\ena)^p, &\quad & i=i_0+1,\dots,r.
 \end{align*}
 \end{enumerate}
\end{corollary}

\begin{remark}
Corollary \ref{cor:specialcases} (i) can be easily proved directly as in~\cite[Lemma~3.23]{Poi-Pul2}. In Corollary \ref{cor:specialcases} (ii), we recover the radii of convergence of the Frobenius pushforward from \cite[Theorem 10.5.1]{pde}. Note that it is not more difficult to get the off-centered version (see \cite[Theorem 10.8.3]{pde}) by using the description of its profile function in Example \ref{ex:Frobenius}. 

To the best of our knowledge, these were the only examples of radii of convergence of pushforward differential equations known so far.
\end{remark}

\subsection{Profile of a differential equation}
One of the most useful properties of the profile function is that it behaves nicely with respect to compositions. More precisely, if we are given morphisms $\vphi:Y\to X$ and $\psi:X\to Z$, and a point $y\in Y^{\hyp}$, then the profile of the morphism $\psi\circ\vphi$ at $y$ is equal to the composition of the profile of the morphism $\psi$ at~$\varphi(y)$ and the profile of the morphism $\vphi$ at $y$. In this section, we reformulate our results on pushforwards of $p$-adic differential equations so that they witness this behavior.

To do so, let $Y$ be a quasi-smooth $k$-analytic curve and let $\ena$ be a differential equation on $Y$ of rank $r$. We will use Notation~\ref{not:dimH0} again.

\begin{defn} 
For $y \in Y^\hyp$, we call \em{profile function of $\ena$ at $y$} the continuous piecewise monomial function $\fb=\fb^y_\ena: [0,1]\to [0,1]$ satisfying $\fb(0) = 0$, $\fb(1)=1$ and
\[\forall s \in (0,1],\ \deg^-_{\fb}(s) = \dim_{k} H^0\big(D(y,s^-),\ena\big).\]
\end{defn}

%\J{$D(y,s^-)$ not defined...}

\begin{remark}\label{rem:E'E''}
 If $\ena = \bigoplus_{i=1}^m (\cE_{i}',\nabla'_{i})$, then we have
\[ \fb^y_{\ena} = \prod_{i=1}^m \fb^y_{(\cE_{i}',\nabla_{i}')}.\]
 In particular, if the $\cE'_{i}$'s are all isomorphic, then 
\[ \fb^y_{\ena} = (\fb^y_{(\cE_{1}',\nabla_{1}')})^m .\]
\end{remark}

Let us now give a more explicit description of this profile function. Let $y\in Y^\hyp$. Let $(\cR_1,\dots,\cR_r)$ be the multiradius of convergence of $\ena$ at $y$ and let $\cR_1'<\dots<\cR_l'$ be the different components of $\cM\cR\big(y,\ena\big)$ (so that we have $\cR_1=\cR_1'$ and $\cR_r=\cR_l'$). Finally, we denote by $r_i$ the number of radii in $\cM\cR\big(y,\ena\big)$ which are equal to $\cR_i'$. For convenience, we set $\cR'_{0}=0$ and $\cR'_{l+1}=1$.

Note that $r_1+\dots+r_l=r$ 
and that, for each $i=0,\dotsc,l$ and each $s \in (\Rc'_{i},\Rc'_{i+1}]$, we have
\[\dim_{k} H^0\big(D(y,s^-),\ena\big) = \sum_{j=i+1}^l r_{j}.\]

It follows that the break-points of the profile function~$\fb$ of~$\ena$ at~$y$ are exactly $\cR'_{0},\dotsc,\cR'_{l+1}$ and that, for each $i=0,\dotsc,l$, we have 
\[\forall s \in [\cR'_{i},\cR'_{i+1}],\ \fb(s) = c_{i} \, s^{d_{i}},\]
where $c_{i} = \prod_{j=i+1}^l {\cR'_{j}}^{-r_{j}}$ and $d_{i} =  \sum_{j=i+1}^l r_{j}$. 
From this description, it is clear that the profile of a differential equation at a point captures its multiradius of convergence at that point, that is, one recovers, from the break-points of the profile and the difference between the corresponding left and right degrees, the components $\cR_1',\dots,\cR_l'$ and their multiplicities, respectively.

It is worth noting that $d_0=r$ and that $c_0=\prod_{i=1}^r \cR_i^{-1}$, which is the numerical invariant used to define irregularity of the differential equation. \\

The following result gives a motivation for introducing the profile of a differential equation. 

\begin{lemma}\label{lem:profiletrivial}
Suppose that $\vphi:Y\to X$ is a finite morphism of quasi-smooth $k$-analytic curves and let $y\in Y^{\hyp}$ such that $\{y\}=\vphi^{-1}(\vphi(y))$. Then, we have 
\[\fb^{\vphi(y)}_{\vphi_*(\sO,d)}=\big((f^y_\vphi)^{-1}\big)^{n_{\vphi}(y)}.\]
\end{lemma}
\begin{proof}
Set $x:= \varphi(y)$. Let~$K$ be an algebraically closed complete valued extension of~$\Hs(y)$. Let $y_{K} \in Y_{K}$ be a $K$-rational point lying over~$y$. Set $x_{K} := \varphi_{K}(y_{K})$. Denote by~$D$ the connected component of~$\pi^{-1}_{K}(x) \setminus \{\sigma_{K}(x)\}$ containing~$x_{K}$. The set $\varphi_{K}^{-1}(D)$ is a disjoint union of connected components $D'_{1},\dotsc,D'_{\ks^y}$ of $\pi^{-1}_{K}(y) \setminus \{\sigma_{K}(y)\}$. Denote by $\psi_{1} : D'_{1} \to D$ the morphism induced by~$\varphi_{K}$.

%We now use the notation of Lemma~\ref{lem:fiber} and Notation~\ref{not:fyNy}.

For each $s\in (0,1]$, we have
\begin{align*}
\dim_{k} H^0\big(D(x,s^-),\vphi_*(\sO,d)\big) &= \dim_{K} H^0\big(D(x_{K},s^-),(\vphi_{K})_*(\sO,d)\big)\\
&= \dim_{K} H^0\big(\varphi_{K}^{-1}(D(x_{K},s^-)),(\sO,d)\big)\\
&= \# \pi_{0}\big(\varphi_{K}^{-1}(D(x_{K},s^-))\big)\\
&=  \ks^y \cdot\#\pi_{0}\big(\psi_{1}^{-1}(D(x_{K},s^-))\big)\\
&= \ks^y \cdot \frac{\ki^y}{\deg^-_{f_{\psi_{1}}}(f_{\psi_{1}}^{-1}(s))}
\end{align*}
by Lemma~\ref{lem:relation} (ii). The result now follows from the fact that we have $\ks^y\, \ki^y = n_{\varphi}(y)$ (see Remark~\ref{rem:resdeghyp}).
\end{proof}

%Set $x:= \varphi(y)$. As usual, extending scalars, we reduce to a computation involving a covering of an open unit disc by  $\ks^y$ open unit discs. Remark \ref{rem:E'E''}  reduces the situation to a covering of an open unit disc centered at $x$ by an open unit disc centered at $y$. Then, for each $s\in (0,1]$, we have
%\begin{align*}
%\dim_{k} H^0\big(D(x,s^-),\vphi_*(\sO,d)\big) &= \dim_{k} H^0\big(\varphi^{-1}(D(x,s^-)),(\sO,d)\big)\\
%&= \# \pi_{0}\big(\varphi^{-1}(D(x,s^-))\big)
%\end{align*}
%and the result follows from Lemma~\ref{lem:relation}.

In this setting, the following result could be considered as an avatar of the pushforward formula from Theorem \ref{push general}. However, we emphasize that its proof does not depend on the pushforward formula, and can be seen as an independent result in its own right.     

\begin{theorem}\label{thm:profile version}
 Let $\vphi:Y\to X$ be a finite \'etale morphism of degree~$d$ between quasi-smooth $k$-analytic curves, let $\ena$ be a differential equation on $Y$ and put $\fna:=\vphi_*\ena$. Let $x\in X^{\hyp}$. 
 Then, we have
 \begin{equation}\label{profileformula} 
 \fb^{x}_{\fna} = \prod_{y \in \varphi^{-1}(x)} \big(\fb^y_\ena \circ (f^y_\vphi)^{-1}\big)^{n_\vphi(y)}.
 \end{equation}
\end{theorem}

\begin{proof}
By Remark~\ref{rem:E'E''}, we may assume that~$\varphi^{-1}(x)$ is a singleton~$\{y\}$. Let us use the same notation as in the proof of Lemma~\ref{lem:profiletrivial}. We have 
\begin{align*}
\deg^-_{\big(\fb^y_\ena \circ (f^y_\vphi)^{-1}\big)^{n_{\vphi}(y)}}(f^y_{\varphi}(s)) &= n_\vphi(y)\cdot\deg^-_{\fb^y_{\ena}}(s) \cdot \deg^-_{(f^y_\vphi)^{-1}}(f^y_{\varphi}(s))\\
&= \dim_{k} H^0\big(D(y,s^-),\ena\big) \cdot \frac{n_{\vphi}(y)}{\deg^-_{f_{\psi_{1}}}(s)}\\
&= \dim_{K} H^0\big(D(y_{K},s^-),\pi^*_{K}\ena\big) \cdot \ks^y_{\vphi}\cdot \# \pi_{0}(\psi_{1}^{-1}\big(D(x_{K},f_{\psi_{1}}(s)^-)\big)\\
&=\dim_{K} H^0\big(D(x_{K},f^y_{\varphi}(s)^-),\pi_{K}^*\fna\big)\\
&=\deg_{\fb^x_{\fna}}^-(f^y_\vphi(s)).
\end{align*}
%where we used Lemma~\ref{lem:relation} and Remark \ref{rmk:pushforward}. 

%By Remark~\ref{rem:E'E''}, we may assume that~$\varphi^{-1}(x)$ is a singleton~$\{y\}$. We note that 
%\begin{align*}
%\deg^-_{\big(\fb^y_\ena \circ (f^y_\vphi)^{-1}\big)^{n_{\vphi}(y)}}(f^y_{\varphi}(s)) &= n_\vphi(y)\cdot\deg^-_{\fb^y_{\ena}}(s) \cdot \deg^-_{(f^y_\vphi)^{-1}}(f^y_{\varphi}(s))\\
%&= \dim_{k} H^0\big(D(y,s^-),\ena\big) \cdot \frac{n_{\vphi}(y)}{\deg^-_{f^y_{\vphi}}(s)}\\
%&= \dim_{k} H^0\big(D(y,s^-),\ena\big) \cdot \ks^y_{\vphi}\cdot \# \pi_{0}(\varphi^{-1}\big(D(x,f^y_{\varphi}(s)^-)\big)\\
%&=\dim_{k} H^0\big(D(x,f^y_{\varphi}(s)^-),\fna\big)\\
%&=\deg_{\fb^x_{\fna}}^-(f^y_\vphi(s)),
%\end{align*}
%where we used Lemma~\ref{lem:relation} and Remark \ref{rmk:pushforward}. 
%
%\J{essentially correct but not quite: points not rational, etc.}

As both sides of \eqref{profileformula} have equal local degrees and they coincide at 0, the result follows.
\end{proof}

\section{Applications}

\subsection{Herbrand function and multiradius of convergence}

In this section, we will establish a close relation between the profile and Herbrand functions of a morphism and the pushforward of the constant connection. 

Let~$\varphi : Y \to X$ be a finite generically \'etale morphism between quasi-smooth $k$-analytic curves. Let $y \in Y^\hyp$ and set $x:=\varphi(y)$. 

\begin{theorem}[\protect{\cite[Theorem~4.5.2]{TemkinHerbrand}}]
The profile function~$f_{\varphi}^y$ coincides with the Herbrand function of the extension $\Hs(y)/\Hs(x)$. 
\qed
\end{theorem}

The fields $\Hs(y)$ and~$\Hs(x)$ are not discretely valued, so the Herbrand function from the previous statement is not the classical one (as in~\cite[Chapitre~IV]{CL} for instance). The one we use here has been defined by M.~Temkin in \cite[Section~4.1]{TemkinHerbrand} and is a sort of multiplicative version of the classical Herbrand function. In this section, we always use M.~Temkin's conventions and notation.

Usually, Herbrand functions are defined for Galois extensions, but since they behave well with respect to extensions, it is possible to define them for separable non-normal extensions, so that the result above makes sense. In the Galois case however, an interpretation in terms of ramification filtration is available. Combined with Lemma~\ref{lem:relation} (ii), we now deduce the following result.

\begin{corollary}
Assume that the extension $\sH(y)/\sH(x)$ is Galois with group~$G$. The real numbers $0 < s^y_{1} < \dotsb < s^y_{n(y)-1} <1$ are the jumps of the upper ramification filtration of the extension $\Hs(y)/\Hs(x)$. 

For each $i=1,\dotsc,n(y)$, we have $n^y_{i} = (G : G^{s^y_{i-1}})$.  
\qed
\end{corollary}

Thanks to Corollary~\ref{push constant connection}, we may now express the radii of convergence of the differential equation $\varphi_{*}(\Os,d)$ in terms of ramification data.

\begin{corollary}\label{cor:Herbrand}
Assume that $\varphi^{-1}(x)=\{y\}$ and that the extension $\sH(y)/\sH(x)$ is Galois with group~$G$. Then, the radii of convergence of the differential equation $\varphi_{*}(\Os,d)$ at~$x$ coincide with the non-zero upper ramification jumps of the extension $\sH(y)/\sH(x)$. 

More precisely, if we denote by $v_{-1} = 1 > v_{0} > \dotsb > v_{n} > v_{n+1} = 0$ the jumps of the upper ramification filtration of the extension~$\sH(y)/\sH(x)$, then we have 
\[\cM\cR(x,\varphi_{*}(\Os,d)) = (v_{n},\dotsc,v_{n},\dotsc,v_{0},\dotsc,v_{0},1),\]
where, for each $j=0,\dotsc,n$, $v_{j}$ appears $(G:G^{v_{j+1}}) - (G:G^{v_{j}})$ times.
\end{corollary}

\begin{remark}
 A higher dimensional analogue of the previous result, that is, a relation between ramification of Galois extensions of higher-dimensional analytic fields and pushfoward of partial $p$-adic differential equations is a task that will be undertaken in the future.
\end{remark}

\subsection{Irregularity, Laplacian and pushforward}

Let $\vphi:Y\to X$ be a finite \'etale morphism of quasi-smooth $k$-analytic curves. Let $y\in Y^{(2)}$ and set $x:=\vphi(y)$. 

For each tangent direction~$\vt$ at~$y$ and each small enough open annulus $A_\vt$ in $Y$ with an endpoint in $y$ that corresponds to $\vt$, the image $A_{\vv}=\vphi(A_\vt)$ is an open annulus that has an endpoint in $x$ (that is, $A_{\vv}$ corresponds to a tangent direction ${\vv}\in T_xX$). In this way we obtain a map of tangent spaces $\vec{\vphi}_{y}:T_yY\to T_xX$. 

Furthermore, $\vphi$ restricts to a finite \'etale morphism  $\vphi_\vt:A_\vt\to A_{\vv}$. If we choose suitable normalizing coordinates $T:A_\vt\iso A(0;q,1)$ and $S:A_{\vv}\iso A(0;q^d,1)$, then~$\vphi_\vt$ can be represented as a function $S=\vphi(T)=T^d(1+h(T))$, where $|h(T)|_\rho<1$ for every $\rho\in (q,1)$. Here $d=d_\vt$ is the degree of $\vphi$ over $A_\vt$. The morphism $\vphi_\vt$ being \'etale implies that the derivative $\frac{dS}{dT}=\vphi'(T)$ does not have zeros in $A_\vt$, so that we may write $\vphi'(T)=aT^\sigma(1+g(t))$, where $\sigma=\sigma_\vt\in\mathbb{Z}$, $a=a_{\vt}$, $0<|a|\leq 1$ and $|g(t)|_\rho<1$ for every $\rho\in (q,1)$. We put
\[
\nu=\nu_\vt:=\sigma-d+1.
\]

\begin{lemma}\label{lem:a1}
 Keep the notation as above and let $\rho\in (q,1) \cap |k^*|$ and $\alpha\in A(0;q,1)(k)$ with $|\alpha|=\rho$.
  Let $D_\alpha$ be the maximal open disc in $A_\vt$ that contains $\alpha$, $T_\alpha: D_\alpha\iso D(0,1^-)$ be a normalizing coordinate sending $\alpha$ to $0$. Define similarly~$D_{\varphi(\alpha)}$ and $T_{\varphi(\alpha)}$. Let $T_{\vphi(\alpha)}=\sum_{i\geq 0}a_iT^i_\alpha$ be the restriction of $\vphi$ over $D_\alpha$ with respect to the  coordinates $T_\alpha$ and $T_{\vphi(\alpha)}$. Then, we have $|a_1|=|a|\rho^\nu$.
\end{lemma}
\begin{proof}
 First of all, we note that for every point $\beta\in A(0;q,1)(k)$ we have $|\frac{dS}{dT}(\beta)|=|\frac{dS}{dT}|_{|\beta|}=|a||\beta|^\sigma$. In particular, for $\beta\in D_\alpha(k)$ we have $|\frac{dS}{dT}(\beta)|=|a|\rho^\sigma$. The last expression does not change if we pass from $S$ and $T$ to any pair of coordinates on $D_{\vphi(\alpha)}$ and $D_\alpha$ that identify the latter discs with $D(0,\rho^{d-})$ and $D(0,\rho^-)$, respectively. In particular, we may choose $T'_{\vphi(\alpha)}:=T_{\vphi(\alpha)}/\vphi(\alpha)$ and $T'_\alpha:=T_\alpha/\alpha$ for such coordinates to obtain 
 \[
 \Big|\frac{\vphi(\alpha)}{\alpha}\Big|\Big|\frac{dT_{\vphi(\alpha)}}{dT_{\alpha}}(\beta)\Big|=\Big|\frac{dT'_{\vphi(\alpha)}}{dT'_{\alpha}}(\beta)\Big|=|a|\rho^\sigma
 \]
  for every $\beta\in D_{\alpha}(k)$. The proof follows by choosing $\beta=\alpha$ and noting that $|\alpha|=\rho$ and $|\vphi(\alpha)|=\rho^d$.
\end{proof}

\begin{remark}\label{rem:nudelta}
 One can show that in the above setting $|a|\rho^\nu=\delta_\rho$, where $\delta_\rho$ is the different of the extension $\sH(\eta_\rho)/\sH(\eta_{\rho^d})$. For this aspect one may refer to \cite[Remark 5.1.1]{BojRH} and \cite[Section 4.2.4]{CTT14}.
 \end{remark}
 
 It follows from \cite[Theorem 3.3.1]{TemkinHerbrand} and its proof that, for $q<1$ close enough to 1, the restriction of the morphism $\vphi_\vt$ to any maximal open disc $D$ in $A_\vt$ will be radial with profile function depending only on the endpoint of $D$. In what follows we will assume that this property holds. Moreover, by further increasing $q$ if necessary we may assume that the profile function is $|k^*|$-monomial over the skeleton of the annulus $A_\vt$ (rather than just piecewise $|k^*|$-monomial; see comments after Theorem \ref{thm:radializingskeleton}).

We keep our setting from above. Let $\ena$ be a differential equation of rank $r$ on $Y$ and let $\fna$ its pushforward on $X$. We denote by $\ena_\vt$ the restriction of $\ena$ to~$A_\vt$ and by $\fna_{\vt}$ the pushforward of $\ena_\vt$ by $\vphi_\vt$ on $A_{\vv}$. 

\begin{thm}\label{thm:pushforward height}
For each $\rho<1$ close enough to~1, we have 
 \[
 h(\eta_{\rho^d},\fna_{\vt})=d\, h(\eta_\rho,\ena_{\vt})-r\nu\log\rho^d-rd\log|a|.
 \]
 In particular, we have
 \begin{equation}\label{eq:pushforward irregularity}
 \Irr_{\vv}(\fna_{\vt})=\Irr_\vt(\ena_\vt)+r\nu.
 \end{equation}
\end{thm}
\begin{proof}
 For a point $\eta_\rho$ on the skeleton of $A_\vt$, let $\vphi_{\vt,\rho}$ be the restriction of $\vphi_\vt$ to the closed annulus $A[0;\rho,\rho]$, let $f_\rho$ be the profile function at $\eta_\rho$, let $0=t^\rho_0<t^\rho_1<\dots<t^\rho_n=1$ be the breakpoints of $f_\rho$ and let $\ks^{\eta_\rho}=\ks^\rho$ be the residual separable degree of $\vphi_{\vt,\rho}$ at~$\eta_\rho$. We also denote by $N_{\rho}$ the function $N_{\vphi_{\vt,\rho},\eta_{\rho^d}}$ and by $0=s^\rho_0<s^\rho_1<\dots<s^\rho_n=1$ its breakpoints. Finally, we set $n^\rho_j:=N_\rho(s^\rho_j)$, for $j=1,\dots,n$. Since $f_\rho$ is monomial over the intervals $(t^\rho_{i-1},t^\rho_i)$, its degree $d^\rho_i$ there will be equal to $d^\rho_n/n^\rho_i$, by Lemma  \ref{lem:relation}. The same lemma  implies that $f(t^\rho_j)=s^\rho_j$. We also note that $d_n$ is equal to the residual inseparable degree~$\ki^{\eta_{\rho}}$ of $\vphi_\vt$ at $\eta_\rho$, by Lemmas \ref{lem:phiyK} and \ref{lem:phiyK34}, so we have $d_n\, \ks^{\eta_\rho}=d$.  We will be using this equality in what follows.  

Let us fix $\rho\in (q,1)$. We drop $\rho$ in $n_i^\rho$, $t_i^\rho$, \dots to lighten the notation while advising the reader to keep in mind the dependence on $\rho$.

Let $\cM\cR(\eta_{\rho},\ena)=(R^\rho_1,\dots,R^\rho_r)=(R_1,\dots,R_r)$ and let us fix $i\in \{1,\dots,r\}$. For some $m$, we have $s_{m}<f_\rho(R_i)\leq s_{m+1}$. Having in mind Theorem \ref{push general}, we now compute the sum of the logarithms of the radii of the pushforward that come from the radius~$R_{i}$, \ie the radii $s_1,\dots,s_1,\dots,s_m,\dots,s_m, f_\rho(R_i),\dots, f_\rho(R_i)$, where each~$s_i$ appears $n_i-n_{i+1}$ times and $f_\rho(R_i)$ appears $n_{m+1}$ times. Multiplying by $\ks$ and summing over~$i$ will give us the desired result.

We use equations \eqref{profileimage} and \eqref{profileimage2} to express $f_\rho(R_i)$ and $s_j$'s in terms of $t_j$'s. By Lemma~\ref{lem:relation}, for $j=0,\dots,n$, we have $n_jd_j=d_n$. Note that, since~$\varphi$ is \'etale, we have $d_{1}=1$, hence $n_{1}=d_n$. We obtain
\begin{align*} 
 & n_{m+1}\log f_\rho(R_i)+\sum_{j=1}^m(n_j-n_{j+1})\log s_j\\
= &n_{m+1}\Big(\log|a_1|+d_{m+1}\log R_i+\sum_{j=1}^m(d_j-d_{j+1})\log t_j\Big)\\
&\quad+\sum_{j=1}^m(n_j-n_{j+1})\Big(\log|a_1|+d_j\log t_j+\sum_{l=1}^{j-1}(d_l-d_{l+1})\log t_l\Big)\\
= & n_{m+1}\log|a_1|+d_n\log R_i+\sum_{j=1}^m(n_{m+1}d_j-n_{m+1}d_{j+1})\log t_j\\
&\quad+\sum_{j=1}^m(n_j-n_{j+1})\log |a_1|+\sum_{j=1}^m(n_jd_j-n_{j+1}d_j)\log t_j\\ 
&\quad+ \sum_{j=1}^m(n_j-n_{j+1})\sum_{l=1}^{j-1}(d_l-d_{l+1})\log t_l\\
= &d_n\log|a_1|+d_n\log R_i+\sum_{j=1}^m\big(n_{m+1}d_j-n_{m+1}d_{j+1}+n_jd_j-n_{j+1}d_j\big)\log t_j\\
&\quad+\sum_{l=1}^{m-1}\Big(\sum_{j=l+1}^m\big(n_jd_l-n_jd_{l+1}-n_{j+1}d_l+n_{j+1}d_{l+1}\big)\Big)\log t_l\\
= &d_n\log|a_1|+d_n\log R_i+\sum_{j=1}^m\big(n_{m+1}d_j-n_{m+1}d_{j+1}+n_jd_j-n_{j+1}d_j\big)\log t_j\\
&\quad+\sum_{l=1}^{m-1}\big(n_{l+1}d_l-n_{l+1}d_{l+1}-n_{m+1}d_l+n_{m+1}d_{l+1}\big)\log t_l\\
= &d_n\log|a_1|+d_n\log R_i.
 \end{align*}

By Lemma~\ref{lem:a1}, we have $|a_1| = |a|\rho^\nu$. Multiplying by $\ks$  and summing over $i=1,\dots,r$, we obtain
\[
h(\eta_{\rho^d},\fna_{\vt})=-rd\log|a|-r\nu\log\rho^d+d\,h(\eta_\rho,\ena_\vt),
\]
which proves the first claim. Taking the derivative with respect to $\log \rho^d$, we obtain the rest.
\end{proof}

\begin{corollary}\label{cor:laplacianxy}
 Suppose that $\vphi^{-1}(x) = \{y\}$ and let $\fna:=\vphi_*\ena$. Let $\Gamma_x$ be a finite subset of $T_xX$ and let $\Gamma_y:=\vec{\vphi}_{y}^{-1}(\Gamma_x)\subset T_yY$. Then
 \begin{equation}\label{laplace}
  \Delta_y(\Gamma_y,\ena)=\Delta_x(\Gamma_x,\fna) + r\sum_{\vt\in \Gamma_y}\nu_\vt.
 \end{equation}
\end{corollary}

\begin{proof}
We note that since~$\Gamma_x$ is a finite subset of $T_xX$, $\Gamma_y$ is also finite, so we may calculate Laplacian of $\ena$ along it.

  Having in mind the additivity of \eqref{laplace} with respect to points in $\Gamma_x$, it is enough to prove the result when $\Gamma_x$ is a singleton, say $\{\vv\}$. Let $\vec{\vphi}_{y}^{-1}(\vv)=\{\vt_1,\dots,\vt_l\}$. For each $i=1,\dots,l$, let $A_{\vt_i}(0;q_i,1)$ be a normalized open annulus in $Y$ such that $\vphi$ restricts to a finite \'etale morphism $\varphi_{i} : A_{\vt_i}(0;q_i,1) \to A_{\vv}(0;q,1)$. By increasing $q$ if necessary, we may assume that, for each $i=1,\dots,l$ and each $\rho\in (q,1)$, the preimage of the point $\eta_\rho\in A_{\vv}(0;q,1)$ by~$\varphi_{i}$ is the point $\eta_{\rho^{\alpha_i}}\in A_{\vt_i}(0;q_i,1)$, where $\alpha_i=d_{\vt_i}^{-1}$, and that the conclusion of Theorem~\ref{thm:pushforward height} holds. We have 
  \begin{align*}
 h\big(\eta_\rho,\fna\big)&=\sum_{i=1}^lh\big(\eta_\rho,\fna_{\vt_i}\big)\\
 &=\sum_{i=1}^l\Big(d_{\vt_i}\,h\big(\eta_{\rho^{\alpha_i}},\ena_{\vt_{i}}\big)-r\nu_{\vt_i}\log\rho-rd_{\vt_i}\log|a_{\vt_{i}}|\Big).
 \end{align*}
 By taking derivatives with respect to $\log\rho$, the result follows.
 \end{proof}

We now derive explicit formulas for Laplacians. Let us first recall the case of~$\mathbb{P}^1$.

\begin{proposition}[\protect{\cite[Proposition~6.2.11]{Poi-Pul3}}]\label{prop:LaplacianP1}
Let~$\fna$ be a differential equation of rank~$r$ on an open subset~$U$ of~$\mathbb{P}^1$. Let $x \in U$ be a point of type~2. Let~$\Gamma_{x}$ be a finite subset of~$T_{x}X$ with at least two elements. Then, we have 
 \[
 \Delta_x(\Gamma_{x},\fna) \le (\# \Gamma_{x}-2)\, r
 \]
 with equality if $\cR_{r}(x,\fna) <1$.
\end{proposition}

We now deduce the general case thanks to the pushforward formula. Let~$X$ be a quasi-smooth $k$-analytic curve and let~$\ena$ be a differential equation of rank~$r$ on~$X$. Recall that, by Theorem~\ref{radius continuity}, for each point $x\in X$, the set
\[\Gamma^{\Irr}_{x} = \{\vt \in T_{x}X \mid  \Irr_{\vt}(\ena) \ne 1\}\]
is finite.

We also recall the local Riemann-Hurwitz formula. Let $X'$ be a quasi-smooth $k$-analytic curve and let $\varphi : X \to X'$ be an \'etale morphism. Let~$x \in X$ be an inner point of type~2 such that $x' := \varphi(x)$ is inner too. Then, by \cite[Theorem~4.5.4]{CTT14} or \cite[Theorem~5.3.1]{BojRH} and Remark~\ref{rem:nudelta}, the set
\[\Gamma^\nu_{x} = \{\vt \in T_{x}X \mid  \nu_{\vt} \ne 1 - d_{\vt}\}\]
is finite and we have
\[2g(x) - 2 = d_{x}\, (2g(x')-2) + \sum_{\vt \in T_{x}X} (\nu_{\vt} + d_{\vt} -1),\]
where~$d_{x}$ denotes the local degree of~$\varphi$ at~$x$.

\begin{corollary}
Let~$x\in X$ be an inner point of type~2. Then, for each finite subset~$\Gamma_{x}$ of~$T_{x} X$ containing~$\Gamma_{x}^\Irr$, we have 
 \[
 \Delta_x(\Gamma_{x},\ena) \le (2g(x)-2+ \# \Gamma_{x})\, r
 \]
 with equality if $\cR_{r}(x,\ena) <1$.
\end{corollary}
\begin{proof}
First note that, if the result holds for some finite subset~$\Gamma_{x}$ of~$T_{x} X$ containing~$\Gamma_{x}^\Irr$,  then it holds for all of them.

Let $\varphi : V \to U$ be an \'etale morphism of degree~$d$ from an open neighborhood of~$x$ in~$V$ to an open subset of~$\mathbb{P}^1$. Set~$x' := \varphi(x)$. We may assume that~$\varphi^{-1}(x')=\{x\}$, so that $d_{x}=d$. Let~$\Gamma'_{x'}$ be a finite subset of~$T_{x'} \mathbb{P}^1$ such that $\Gamma_{x} := \varphi^{-1}(\Gamma'_{x'})$ contains $\Gamma_{x}^\Irr \cup \Gamma^\nu_{x}$.

By the local Riemann-Hurwitz formula, we have
\[\sum_{\vt \in \Gamma_{x}} (\nu_{\vt} + d_{\vt} -1) = \sum_{\vt \in T_{x}X} (\nu_{\vt} + d_{\vt} -1) = 2g(x)-2+2d.\]
Moreover, we have 
\[\sum_{\vt \in \Gamma_{x}} d_{\vt} = \sum_{\vv \in \Gamma'_{x}} \sum_{\vt \in \vec\varphi_{x}^{-1}(\vv)} d_{\vt} = d \cdot \# \Gamma'_{x'},\] 
hence 
\[\sum_{\vt \in \Gamma_{x}} (d_{\vt}-1) = d \cdot \# \Gamma'_{x'} - \# \Gamma_{x}\]
and 
\[\sum_{\vt\in \Gamma_{x}} \nu_{\vt} =  2g(x)-2+2d - d \cdot \# \Gamma'_{x'} + \# \Gamma_{x}.\]
Noting that the rank of $\varphi_{*}\ena$ is $rd$, the result now follows from Corollary~\ref{cor:laplacianxy} and Proposition~\ref{prop:LaplacianP1}, together with Theorem~\ref{push general} to show that $\cR_{rd}(x',\varphi_{*}\ena)<1$ when $\cR_{r}(x,\ena)<1$.
\end{proof}

By a classical decomposition argument, a more general result follows. We refer to the proof of \cite[Proposition~6.2.20]{Poi-Pul3} for details.

For $x\in X^\hyp$ and $i=1,\dotsc,r$, we define an $i^\textrm{th}$ Laplacian $\Delta^i_x(\Gamma_{x},\ena)$ by considering the slopes of the $i^\textrm{th}$ partial height of the Newton polygon instead of that of its total height. Denote by~$\Gamma^{\cR}_{x}$ the subset of~$T_{x}X$ defined by the condition that, for each $\vt \in \Gamma^{\cR}_{x}$, there exists $j\in \{1,\dotsc,r\}$ such that the slope of the $j^\textrm{th}$ radius of convergence along~$\vt$ is different from~1. By Theorem~\ref{radius continuity}, it is finite.

\begin{corollary}\label{cor:partialheight}
Let~$x\in X$ be an inner point of type~2. Let $i\in\{1,\dotsc,r\}$. Then, for each finite subset~$\Gamma_{x}$ of~$T_{x} X$ containing~$\Gamma_{x}^\cR$, we have 
 \[
 \Delta^i_x(\Gamma_{x},\ena) \le (2g(x)-2+ \# \Gamma_{x})\, i
 \]
 with equality in the following two cases:
 \begin{enumerate}[a)]
 \item $i=r$ and $\cR_{r}(x,\ena)<1$;
 \item $i<r$ and $\cR_{i}(x,\ena) < \cR_{i+1}(x,\ena)$.
 \end{enumerate}
\end{corollary}

\begin{rmk}
This result appears in \cite[Theorem~5.3.6]{Kedlayalocalglobal}, although with some details left out in the proof and using a quite dense argument. 

In \cite[Proposition~6.2.20]{Poi-Pul3}, the formula was proved under some assumptions, which are satisfied when $g(x)=0$ or $\textrm{char}(\tilde k) \ne 2$. The idea is to construct a morphism~$\varphi$ as in the proof such that the residue morphism $\tilde\varphi_{x} : \sC_{x} \to \sC_{x'}$ is everywhere tamely ramified. When $\textrm{char}(\tilde k) \ne 2$, a result of Fulton shows that this is always possible (see \cite[Proposition~8.1]{Fulton}). 

Here, we prove the formula in full generality and more directly, without relying on heavy geometric input.
\end{rmk}

\begin{rmk}
In the proof of the above results we used the local Riemann-Hurwitz formula. We indicate that, conversely, this formula can be deduced from Theorem \ref{thm:pushforward height} and Corollary \ref{cor:laplacianxy}, combined with a suitable $p$-adic index theorem that relates the Euler-Poincar\'e characteristic of a $p$-adic differential equation to that of the curve itself \textit{via} the irregularities of the equation. This aspect will be further explored in some other paper. 
\end{rmk}


\begin{thebibliography}{BGR84}

\bibitem[Bal10]{Bal10}
F.~Baldassarri.
\newblock {Continuity of the radius of convergence of differential equations on
  $p$-adic analytic curves.}
\newblock {\em Invent. Math.}, 182(3):513--584, 2010.

\bibitem[Bal13]{BaldaMilano}
F.~Baldassarri.
\newblock Radius of convergence of {$p$}-adic connections and the {$p$}-adic
  {R}olle theorem.
\newblock {\em Milan J. Math.}, 81(2):397--419, 2013.

\bibitem[Ber90]{Ber90}
V.~G.~Berkovich.
\newblock {\em Spectral theory and analytic geometry over non-{A}rchimedean
  fields}, volume~33 of {\em Mathematical Surveys and Monographs}.
\newblock American Mathematical Society, Providence, RI, 1990.

\bibitem[BGR84]{BGR}
S.~Bosch, U.~G{\"u}ntzer and R.~Remmert.
\newblock {\em Non-{A}rchimedean analysis}, volume 261 of {\em Grundlehren der
  Mathematischen Wissenschaften}.
\newblock Springer-Verlag, Berlin, 1984.
\newblock A systematic approach to rigid analytic geometry.

\bibitem[BK]{Bal-Ked}
F.~Baldassarri and K.~S.~Kedlaya.
\newblock Harmonic functions attached to meromorphic connections on
  non-archimedean curves.
\newblock In preparation.

\bibitem[Boj17]{BojRH}
V.~Bojkovi\'c.
\newblock Riemann-{H}urwitz formula for finite morphisms of $p$-adic curves.
\newblock {\em Math. Zeit.}, 2017.
\newblock To appear.

\bibitem[Bos69]{Orthonormalbasen}
S.~Bosch.
\newblock {Orthonormalbasen in der nichtarchimedischen Funktionentheorie.}
\newblock {\em Manuscr. Math.}, 1:35--57, 1969.

\bibitem[CTT16]{CTT14}
A.~Cohen, M.~Temkin and D.~Trushin.
\newblock {Morphisms of Berkovich curves and the different function.}
\newblock {\em {Adv. Math.}}, 303:800--858, 2016.

\bibitem[Duc14]{Duc-book}
A.~Ducros.
\newblock La structure des courbes analytiques.
\newblock Manuscript available at
  \url{http://www.math.jussieu.fr/~ducros/trirss.pdf}, 2014.

\bibitem[Ful69]{Fulton}
W.~Fulton.
\newblock Hurwitz schemes and irreducibility of moduli of algebraic curves.
\newblock {\em Ann. of Math. (2)}, 90:542--575, 1969.

\bibitem[Ked10]{pde}
K.~S.~Kedlaya.
\newblock {\em {p}-adic differential equations}, volume 125 of {\em Cambridge
  Studies in Advanced Mathematics}.
\newblock Cambridge University Press, Cambridge, 2010.

\bibitem[{Ked}15]{Kedlayalocalglobal}
K.~S.~Kedlaya.
\newblock {Local and global structure of connections on nonarchimedean curves.}
\newblock {\em {Compos. Math.}}, 151(6):1096--1156, 2015.

\bibitem[Lut37]{Lut37}
\'E.~Lutz.
\newblock Sur l'{\'e}quation $y^2= x^3-{A}x-{B}$ dans les corps $p$-adiques.
\newblock {\em J. Reine Angew. Math.}, 177:238--247, 1937.

\bibitem[Poi13]{Angie}
J.~Poineau.
\newblock Les espaces de {B}erkovich sont ang\'eliques.
\newblock {\em Bull. Soc. Math. France}, 141(2):267--297, 2013.

\bibitem[PP13]{Poi-Pul3}
J.~Poineau and A.~Pulita.
\newblock The convergence {N}ewton polygon of a $p$-adic differential equation
  {III}: {G}lobal decomposition and controlling graphs.
\newblock arXiv, 2013.
\newblock \url{http://arxiv.org/abs/1308.0859}.

\bibitem[PP15]{Poi-Pul2}
J.~Poineau and A.~Pulita.
\newblock {The convergence Newton polygon of a $p$-adic differential equation.
  II: Continuity and finiteness on Berkovich curves.}
\newblock {\em {Acta Math.}}, 214(2):357--393, 2015.

\bibitem[{Pul}15]{Pul}
A.~Pulita.
\newblock {The convergence Newton polygon of a $p$-adic differential equation.
  I: Affinoid domains of the Berkovich affine line.}
\newblock {\em {Acta Math.}}, 214(2):307--355, 2015.

\bibitem[Ser68]{CL}
J.-P.~Serre.
\newblock {\em Corps locaux}.
\newblock Hermann, Paris, 1968.
\newblock Deuxi\`eme \'edition, Publications de l'Universit\'e de Nancago, No.
  VIII.

\bibitem[Tem10]{stablemodification}
M.~Temkin.
\newblock Stable modification of relative curves.
\newblock {\em J. Algebraic Geom.}, 19(4):603--677, 2010.

\bibitem[Tem17]{TemkinHerbrand}
M.~Temkin.
\newblock Metric uniformization of morphisms of {B}erkovich curves.
\newblock {\em Advances in Mathematics}, 317:438 -- 472, 2017.

\end{thebibliography}
\end{document}